\documentclass[11pt]{article}

\usepackage{amsmath,amsthm,verbatim,amssymb,amsfonts,amscd, graphicx}
\usepackage{graphics}
\theoremstyle{plain}

\newtheorem{question}{Question} 
\newtheorem{thm}{Theorem}
\newtheorem{defn}{Definition}[section]
\newtheorem{cor}[defn]{Corollary}

\newtheorem{prop}[defn]{Proposition}
\newtheorem{remark}[defn]{Remark}
\newtheorem{lm}[defn]{Lemma}

\newtheorem{ep}[defn]{Example}
\usepackage{extarrows}
\usepackage{mathrsfs}
\usepackage{galois}
\usepackage[colorlinks,citecolor = red, linkcolor=blue,hyperindex,CJKbookmarks]{hyperref}
\usepackage[all]{hypcap}
\usepackage[all]{xy}
\usepackage{tikz-cd}
\usepackage{latexsym}

\begin{document}
	
	\title{Extensions of immersions of surfaces into $\mathbb{R}^{3}$}
	\author{Bojun Zhao}
	\maketitle
	
	ABSTRACT: 
	This paper is to study the $\mathbb{R}^{3}$ case of \hyperref[Zhao]{[9]}.
	We determine all equivalence classes of immersed $3$-manifolds bounded by an arbitrary immersed surface in $\mathbb{R}^{3}$.

\section{Introduction}\

In this paper,
we assume all $3$-manifolds are oriented,
and the $3$-manifolds will be connected if not otherwise mentioned.
We assume all immersions are transverse immersions,
and all graphs have no isolated point.
We work in PL category: all $3$-manifolds are assumed to have a PL structure, and all maps (between $3$ manifolds) are assumed to be PL maps.

Fix a closed oriented surface $\Sigma$ and an immersion $f: \Sigma \to \mathbb{R}^{3}$.
We say an immersion $F: M \to \mathbb{R}^{3}$ ($M$ a compact, connected $3$-manifold with boundary $\Sigma$) \emph{extends} $f$
if $F \mid_{\partial M} = f$ (toward the side that inward normal vectors point to).

\begin{defn}\rm
	Let $\Sigma$ be a closed oriented surface and $f: \Sigma \to \mathbb{R}^{3}$ an immersion.
	Assume $g_1: M_1 \to \mathbb{R}^{3},g_2: M_2 \to \mathbb{R}^{3}$ are $2$ extensions of $f$.
	$g_1,g_2$ are \emph{equivalent} if there exists a (PL) homeomorphism $h: M_1 \to M_2$
	such that $g_1 = g_2 \circ h$.
	(see \hyperref[Pappas]{[7, Section 2]}, while it states this definition in smooth category)
\end{defn}

\begin{question}\label{question 1}\rm 
	Which immersed closed oriented surfaces in $\mathbb{R}^{3}$ bound immersed $3$-manifolds, and in how many inequivalent ways?
\end{question}

The $2$-dimensional problems were solved by S. Blank (\hyperref[Blank]{[3]}, for immersed disks bounded by the immersed planar circle), K. Bailey (\hyperref[Bailey]{[2]}, for immersed surfaces bounded by the immersed planar circle).
But their algebraic approach does not readily generalize (see \hyperref[Pappas]{[7, Section 1]}).

In \hyperref[Zhao]{[9]} we presented the technique in $2$-dimensional case.
We answer Question \ref{question 1} in this paper:
Given an immersed surface in $\mathbb{R}^{3}$,
we determine all equivalence classes of immersed $3$-manifolds bounded by it (Theorem \ref{main theorem}).

The following question provides the author the basic motivation to accomplish this paper. 
It includes the request to determine the equivalence classes of immersed $3$-balls bounded by the immersed $2$-spheres.

\begin{question}\label{question 2}\rm \hyperref[Kirby]{[5, Problem 3.19]}
	Which immersed $2$-spheres in $\mathbb{R}^{3}$ bound immersed $3$-balls?
\end{question}

By applying the algorithm \hyperref[Preaux]{[8]} after determining all inequivalent immersed $3$-manifolds bounded by an immersed $2$-sphere,
we can determine all inequivalent immersed $3$-balls bounded by the immersed $2$-sphere (Corollary \ref{immersed ball}).

\subsection{Main results}\

Fix a closed oriented surface $\Sigma$ and an immersion $f: \Sigma \to \mathbb{R}^{3}$.
$f$ can't extend if there exists $x \in \mathbb{R}^{3} \setminus f(\Sigma)$ such that $\omega(f,x) < 0$ (where $\omega(f,x)$ denotes the winding number of $f$ around $x$, see Deginition \ref{winding number}).
If $\omega(f,x) \geqslant 0$ for every $x \in \mathbb{R}^{3} \setminus f(\Sigma)$,
an inscribed set $\zeta$ of $f$ (Definition \ref{inscribed set}) is a finite set, 
and $I(\zeta)$ (Definition \ref{good complex set}) is a subset of $\zeta$ ($\zeta,I(\zeta)$ exist, and they can be obtained in finite steps).

\begin{thm}\label{main theorem}
	For a closed oriented surface $\Sigma$,
	let $f: \Sigma \to \mathbb{R}^{3}$ be an immersion such that $\omega(f,x) \geqslant 0, \forall x \in \mathbb{R}^{3} \setminus f(\Sigma)$.
	Assume $\zeta$ is an inscribed set of $f$.
	Then there exists a bijection between $I(\zeta)$ and all equivalence classes of immersions of $3$-manifolds to extend $f$.
\end{thm}

\hyperref[Preaux]{[8]} (or, see \hyperref[AFW]{[1, Section 4, C.29, C. 30]})
provides an algorithm to detect if a $3$-manifold with boundary $S^{2}$ is a $3$-ball.
We apply this to determine the immersed $3$-balls in an immersed $2$-sphere:

Assume $\Sigma = S^{2}$.
Assume $E(f)$ is the set of equivalence classes of immersions of $3$-manifolds to extend $f$.
Then Theorem \ref{main theorem} gives a bijection $q: I(\zeta) \to E(f)$.
For each $A \in I(\zeta)$,
choose $g_A: M_A \to \mathbb{R}^{3}$ an extension to represent the equivalence class $q(A) \in E(f)$.
Definition \ref{inscribed map} provides $M_A$ a trangulation (determined by $A$).
By applying \hyperref[Preaux]{[8]},
we can detect if $M_A$ is a $3$-ball.
Hence we can detect $I_0(\zeta) = \{A \in I(\zeta) \mid M_A \cong B^{3}\}$.

\begin{cor}\label{immersed ball}
	Let $f: S^{2} \to \mathbb{R}^{3}$ be an immersion such that $\omega(f,x) \geqslant 0, \forall x \in \mathbb{R}^{3} \setminus f(\Sigma)$.
	Assume $\zeta$ is an inscribed set of $f$.
	Then there is a bijection between $I_0(\zeta)$ and all equivalence classes of immersions of $3$-balls to extend $f$.
\end{cor}

\subsection{Organization}\

We will give some basic definitions in Subsection \ref{subsection 2.1},
and we will introduce the branched immersion, good $2$-complexes, cancellation operation in Subsection \ref{subsection 2.2}, Subsection \ref{subsection 2.3}, Subsection \ref{subsection 2.4}.
In Section \ref{section 3},
we will define the $(M,G)$-simple $2$-complex in a compact $3$-manifold $M$ (with nonempty boundary) with a (trivalent) embedded graph $G \subseteq \partial M$,
and we will give the way to construct it.
In Section \ref{section 4},
we will define the inscribed set.
In Section \ref{section 5},
we will prove Theorem \ref{main theorem}.
	
\section{Preliminaries}\label{section 2}\

In this section,
we introduce some basic ingredients.

\subsection{The immersed surfaces in $\mathbb{R}^{3}$}\label{subsection 2.1}\

\begin{defn}\label{winding number}\rm
	(Winding number)
	Let $\Sigma$ be a closed oriented surface and $f: \Sigma \to \mathbb{R}^{3}$ an immersion.
	Chosen $x \in \mathbb{R}^{3} \setminus f(\Sigma)$,
	assume $u: \Sigma \to S^{2}$ is the map such that $u(t) = \dfrac{f(t) - x}{|f(t) - x|}$ ($\forall t \in \Sigma$).
	Let $\omega(f,x) = deg_{u}(x)$ (see \hyperref[GP]{[4, Page 144]}).
\end{defn}\rm

\begin{remark}\rm
	If $F: M \to \mathbb{R}^{3}$ is an immersion to extend $f$,
	then $\omega(f,x)$ is the number of preimages under $F$ at each $x \in \mathbb{R}^{3} \setminus f(\Sigma)$.
	In the rest of this paper,
	if $f: \Sigma \to \mathbb{R}^{3}$ is an immersion,
	we will always assume that 
	$\omega(f,x) \geqslant 0, \forall x \in \mathbb{R}^{3} \setminus f(\Sigma)$
	(if not,
	then there is no immersed $3$-manifold to extend $f$).
\end{remark}

\begin{defn}\rm\label{D_n(f)}
	Let $\Sigma$ be a closed oriented surface and $f: \Sigma \to \mathbb{R}^{3}$ an immersion.
	For each $1 \leqslant k \leqslant \max_{x \in \mathbb{R}^{3} \setminus f(\Sigma)} \omega(f,x)$,
	let $D_k(f) = \overline{\{x \in \mathbb{R}^{3} \setminus f(\Sigma) \mid
	\omega(f,x) \geqslant k\}}$.
    And we let $G_k(f) = \partial D_k(f) \cap \partial D_{k-1}(f)$ ($2 \leqslant k \leqslant \max_{x \in \mathbb{R}^{3} \setminus f(\Sigma)} \omega(f,x)$)
    and $G_1(f) = \emptyset$.
\end{defn}

\begin{figure}\label{xyz}
	\includegraphics[width=0.3\textwidth]{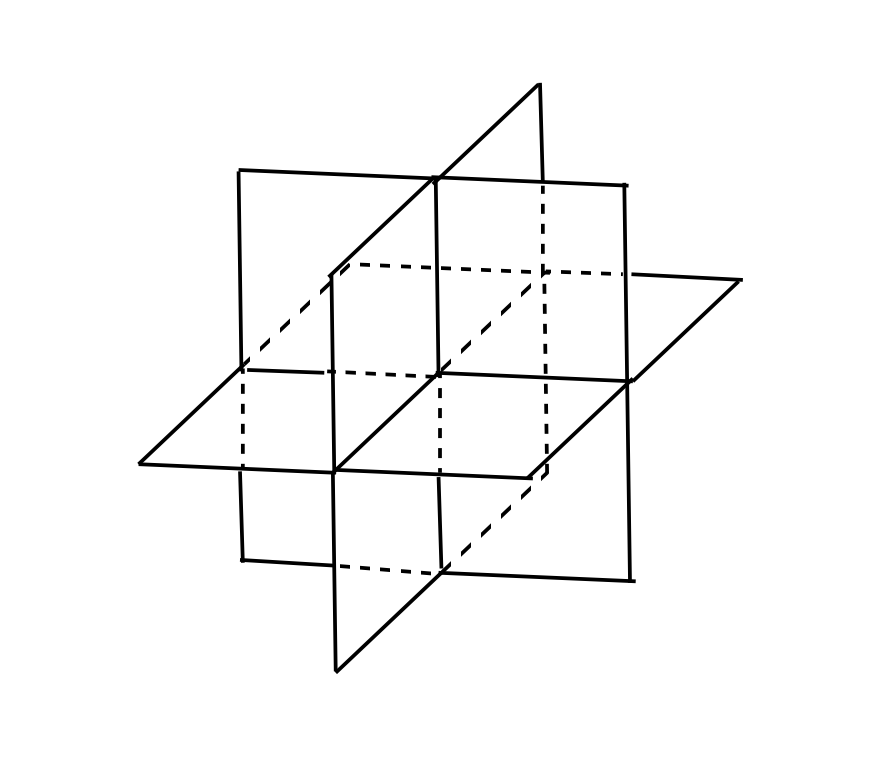}
	\includegraphics[width=0.7\textwidth]{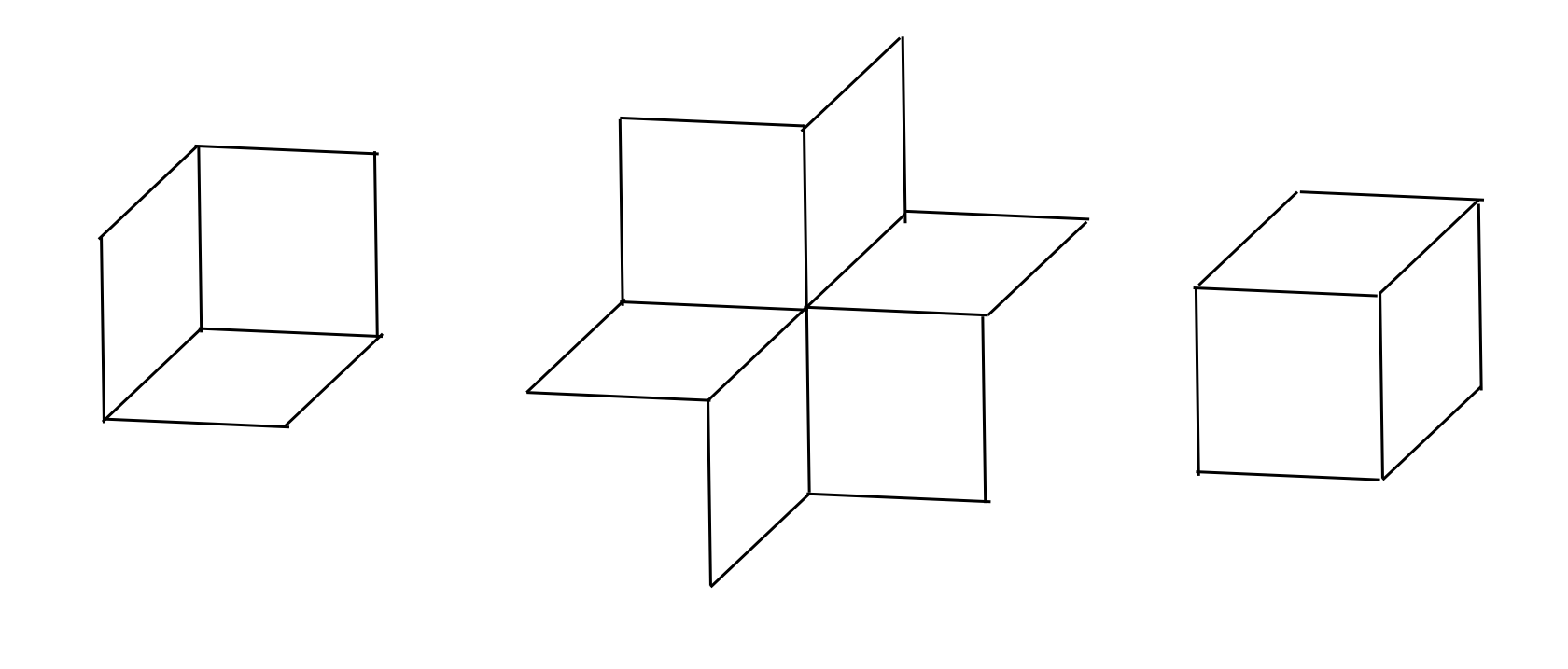}
	\caption{$\partial D_{k-1}(f),\partial D_k(f),\partial D_{k+1}(f)$ intersect at one point (a triple point).}
\end{figure}

Both Definition \ref{winding number} and Definition \ref{D_n(f)} can be generalized to the case of an immersion of a disconnected surface. We will apply this in Subsection \ref{subsection 2.4}.

In \hyperref[LV]{[6]},
if $f: \Sigma \to \mathbb{R}^{3}$ is a (transverse) immersion of a closed oriented surface $\Sigma$,
the points in $f(\Sigma)$ with $1,2,3$ preimages are called \emph{simple points}, \emph{double points}, \emph{triple points}.
The non-simple points, triple points of $f(\Sigma)$ are denoted by
$S(f(\Sigma))$, $T(f(\Sigma))$.

Obviously,
$G_k(f) \subseteq S(f(\Sigma)) \cap \partial D_{k}(f) = G_k(f) \cup G_{k+1}(f)$.
Actually,
$G_k(f)$ is an embedded graph such that all vertices have degree $2$ or $3$ (in this paper, we assume that all embedded graphs have no isolated point), 
and $\{v \in V(G_k(f)) \mid deg_{G_k(f)}(v) = 3\} = G_k(f) \cap T(f(\Sigma))$.
We will not emphasize this in the rest of this paper.

Figure \ref{xyz} shows how $\partial D_{k-1}(f),\partial D_k(f),\partial D_{k+1}(f)$ intersect at a tripe point.
To describe the relation between $G_k(f)$ and $G_{k+1}(f)$ in $\partial D_k(f)$ (see the third picture in Figure \ref{xyz}),
we give the following statement:

\begin{defn}\rm\label{thin trivalent graph}
	Let $\Sigma$ be a closed oriented surface and $G,G^{'} \subseteq \Sigma$ an embedded graphs such that all vertices have degree $2$ or $3$.
	$G^{'}$ is a \emph{thin trivalent graph} of $G$ in $\Sigma$ if:
	
	$\bullet$
	For each $x \in G^{'} \cap G$,
	$x \in \{v \mid v \in V(G), deg_G(v) = 3\}$
	and $x \in \{v \mid v \in V(G^{'}), deg_{G^{'}}(v) = 3\}$.
	Assume $a,b,c,d,e,f$ are the $6$ edges of $G$ and $G^{'}$ at $x$ clockwise,
	and $a \in E(G)$.
	Then $c,e \in E(G)$, $b,d,f \in E(G^{'})$.
\end{defn}

\subsection{Branched immersion}\label{subsection 2.2}\

A (compact) topological space is a \emph{polyhedron} if it is the underlying space of a simplicial complex.
In this paper,
we say a polyhedron $K$ is a \emph{branched $3$-manifold} if
there exists $M$ a compact oriented $3$-manifold and $S_1,\ldots,S_n$ some components of $\partial M$, 
$S_1, \ldots, S_n \ncong S^{2}$ (and we allow $\{S_1,\ldots,S_n\} = \emptyset$),
$K = M \cup_{i_1} C(S_1) \cup_{i_2} \ldots \cup_{i_n} C(S_n)$
(where $i_1,\ldots,i_n$ are the identity maps of $S_1,\ldots,S_n$,
and $C(S) = S \times I / S \times \{1\}$ for an arbitrary topological space $S$).
Moreover,
we denote $\partial M \setminus (S_1 \cup \ldots \cup S_n)$ by $\partial K$ and say it is \emph{boundary} of $K$.
And we denote $\{the$ $vertices$ $of$ $the$ $cones$ $C(S_1),\ldots,C(S_n)\}$ by $B(K)$
(i.e. $B(K) = \{the$ $points$ $in$ $K$ $that$ $have$ $no$ $open$ $neighborhood$ $homeomorphic$ $to$ $\mathbb{R}^{3}$ $or$ $\mathbb{R}^{3}_{+}\}$).

The following statement generalize the branched covers to the map of a branched $3$-manifold $K$ into $\mathbb{R}^{3}$ (we request $B(K)$ to lie in the singular set, then $K \setminus B(K)$ is still a noncompact $3$-manifold).

\begin{defn}\rm
	Let $K$ be a branched $3$-manifold and $g: K \to \mathbb{R}^{3}$ a PL continuous map.
	$g$ is called a \emph{branched immersion} if there exists $F \subseteq \mathbb{R}^{3}$ an embedded graph such that
	$g^{-1}(F)$ is an embedded graph,
	$B(K) \subseteq g^{-1}(F)$,
	and $g \mid_{K \setminus g^{-1}(F)}$ is an immersion. 
	The \emph{singular set} of $g$ is the set consisting of all $x \in K$ such that $g$ is not a locally homeomorphism at $x$,
	and the \emph{branch set} of $g$ is the image of singular set under $g$.
\end{defn}

\begin{remark}\rm
In this paper,
if $g: K \to \mathbb{R}^{3}$ is a branched immersion of a branched $3$-manifold $K$
and $S$, $B$ are the singular set, branch set of $g$,
we will always assume that $g$ maps $S$ homeomorphically to $B$.
For each branch point $y \in B$,
assume $\{x\} = g^{-1}(y) \cap S$,
we say $y$ has index $k$ if $g$ is $(k+1)$-to-one near $x$.
\end{remark}

\begin{remark}\rm
	We explain the difference between the branched covers and our definition (branched immersion):
	we do not request it to be proper;
	we allow $x \in K$ a singular point whose link with respect to $K$ is a not a $2$-sphere,
	then $g \mid_{lk(x,K)}$ is a branched cover of a surface to a $2$-sphere (the number of such points is finite in total).
	
	Actually,
	different from constructing $3$-manifolds from branched covers that branched over links,
	the maps that branched over embedded graphs may construct branched $3$-manifolds.
	That's why we define the branched immersions in branched $3$-manifolds.
	We will introduce the cancellation operation in Subsection \ref{subsection 2.4},
	which is defined in the branched immersions of branched $3$-manifolds.
	
	For a branched immersion $g: K \to \mathbb{R}^{3}$ (where $K$ is a branched $3$-manifold, and $K$ is not a $3$-manifold),
	$g$ can be transformed to a branched immersion of a $3$-manifold into $\mathbb{R}^{3}$ by deleting an open neighborhood at each $x \in B(K)$
	and filling a handlebody.
	But this branched immersion does not send the singular set homeomorphically to the branch set (also, this branched immersion is not an open map).
	So we do not do such transformation.
\end{remark} 

\begin{ep}\rm
	Assume $C(T^{2}) = T^{2} \times I / T^{2} \times \{1\}$ is a cone of a torus,
	and $B^{3} = S^{2} \times I / S^{2} \times \{1\}$ is a $3$-ball in $\mathbb{R}^{3}$.
	Let $p: T^{2} \to S^{2}$ be an arbitrary branched cover.
	Let $g:C(T^{2}) \to B^{3}$ ($B^{3} \subseteq \mathbb{R}^{3}$) be the map such that 
	$g(x,t) = (p(x),t)$ ($\forall x \in T^{2}, t \in [0,1)$),
	$g(T^{2},1) = (S^{2},1)$.
	Then $g$ is a branched immersion.
\end{ep}

\subsection{Good $2$-complexes}\label{subsection 2.3}\

\begin{defn}\rm \label{good complex}
	Let $M$ be a compact $3$-manifold with nonempty boundary
	and $G \subseteq \partial M$ an embedded graph such that all vertices have degree $2$ or $3$.
	Let $X \subseteq M$ be an embedded $2$-complex.
	We say $X$ is a \emph{good $2$-complex in $M$ with respect to $G$} if:
	
	$\bullet$
	Let $\dot{\varphi}^{2}_{X}: \coprod_\alpha \partial D_{\alpha}^{2} \to X^{1}$
	be the attaching map of all $2$-cells of $X$.
	Then $\dot{\varphi}^{2}_{X}$ is surjective. (i.e. all points in $X$ have local dimension $2$)
	
	$\bullet$
	For each $2$-cell $e_{\alpha}$ of $X$,
	the characteristic map $\varphi^{2}_{\alpha}: D_{\alpha}^{2} \to X$ is an embedding.
	
	$\bullet$
	$X \cap \partial M = G$, 
	and $G \setminus \{v \in V(G) \mid deg_G(v) = 3\}$ is the set consisting of
	all $t \in X$ such that $\exists N(t)$ an open neighborhood of $t$ in $M$,
	$(N(t) \cap X,t) \cong (\mathbb{R}^{2}_{+},0)$.
	
	$\bullet$
	For each $t \in \{v \in V(G) \mid deg_G(v) = 3\}$,
	$t$ has an open neighborhood $N(t)$ in $M$ such that $(N(t) \cap X,t)$
	is homeomorphic to 
	$(\{x = 0, y \geqslant 0, z \geqslant 0\} \cup \{y = 0, z \geqslant 0\},0)$ (see Figure \ref{points in simple complex} (a), where $\{x = 0, y \geqslant 0, z \geqslant 0\} \cup \{y = 0, z \geqslant 0\}$ denotes a subset of $\mathbb{R}^{3}$, $(x,y,z)$ is the coordinates of $\mathbb{R}^{3}$).
\end{defn}

Since the set of non-simple points in an immersed surface $f(\Sigma)$ ($f: \Sigma \to \mathbb{R}^{3}$ is an immersion of a surface $\Sigma$) is denoted by $S(f(\Sigma))$,
we generalize this notation to an arbitrary $2$-complex:

\begin{defn}\rm
	For an arbitrary $2$-complex $X$,
	we denote by $S(X)$ the set consisting of all points in $X$ that have no open neighborhood in $X$ homeomorphic to $\mathbb{R}^{2}$ or $\mathbb{R}^{2}_{+}$.
\end{defn}

\subsection{Cancellation operation}\label{subsection 2.4}\

\hyperref[Zhao]{[9]}
states the cancellation operation for a polymersion of a surface (with nonempty boundary) into a surface.
In this subsection,
we define the cancellation operation for a branched immersion of a branched $3$-manifold (with nonempty boundary) into $\mathbb{R}^{3}$.

Recall that Definition \ref{winding number} and Definition \ref{D_n(f)} can be generalized to the case of an immersion of a disconnected surface.
Assume $K$ is a branched $3$-manifold with nonempty boundary ($K$ may be disconnected),
and $g: K \to \mathbb{R}^{3}$ is a branched immersion.
Assume $n = \max_{x \in \mathbb{R}^{3} \setminus g(\partial K)} \omega(g \mid_{\partial K},x)$.
We denote $D_n(g \mid_{\partial K}), G_n(g \mid_{\partial K})$
by $R(g), G(g)$.

\begin{defn}[Cancellable domains]\label{cancellable components}\rm
	Let $K$ be a branched $3$-manifold with nonempty boundary ($K$ may be disconnected)
	and $g: K \to \mathbb{R}^{3}$ a branched immersion.
	
	(i)
	Assume $A_1, A_2, \ldots, A_n \subseteq K$ are closed domains (in this paper, the ``domains'' in the space are compact connected co-dimension $0$ submanifolds).
	$A_1, A_2, \ldots, A_n$ are called \emph{cancellable} if:
	
	$\bullet$
	$Int(A_1),$ $Int(A_2), \ldots, Int(A_n)$ are homeomorphically embedded into $R(g)$ by $g$.
	
	$\bullet$
	There exists $X$ a good $2$-complex in $R(g)$ with respect to $G(g)$
	such that $\{g(Int(A_1)),$ $g(Int(A_2)),$ $\ldots, g(Int(A_n))\}
	= \{the$ $components$ $of$ $Int(R(g)) \setminus X\}$
	(i.e.  $g$ maps $A_1, A_2, \ldots, A_n$ homemomorphically to the closed components 
	obtained by cutting off $X$ from $R(g)$).
	We call $X$ the $2$-complex \emph{associated} to $A_1, A_2, \ldots, A_n$.
	
	$\bullet$
	$(g \mid_{A_i})^{-1}(g(A_i)\cap \partial R(g)) \subseteq \partial K$ if $g(A_i) \cap \partial R(g) \ne \emptyset$
	($\forall i \in \{1,2,\ldots,n\}$).
	
	(ii)
	 We denote $g(\overline{\partial(A_1 \cup A_2 \cup \ldots \cup A_n) \setminus \partial K})$ by $X(A_1,A_2,\ldots,A_n)$.
\end{defn}

Obviously,
$X(A_1,A_2,\ldots,A_n)$ is a subcomplex of $X$,
and it is a good $2$-complex in $R(g)$ with respect to $G(g)$.

\begin{remark}\label{determine cancellable domains}\rm
The cancellable domains $A_1,\ldots,A_n$ can be determined uniquely in following $2$ cases:

(a) 
If each component of $R(g) \setminus X$ contains a component of $\partial R(g) \setminus G(g)$,
then $A_1,\ldots,A_n$ are determined uniquely by $X$

(b)
Fix the associated $2$-complex $X$.
Given a set $P \subseteq K$ 
such that $A_i \cap (\partial K \cup P) \ne \emptyset$ ($\forall i \in \{1,2,\ldots,n\}$),
then $A_1,\ldots,A_n$ are determined uniquely.
\end{remark}

\begin{defn}[Cancellation operation]\label{cancellation operation}\rm
	Let $K$ be a branched $3$-manifold with nonempty boundary ($K$ may be disconnected)
	and $g: K \to \mathbb{R}^{3}$ a branched immersion.
	Assume that the closed domains $A_1, \ldots, A_n \subseteq K$ are cancellable,
	and $X$ is the $2$-complex associated to $A_1,\ldots,A_n$.
	
	(i)
	A \emph{cancellation of $\{A_1,\ldots,A_n\}$} (\emph{canceling $\{A_1,\ldots,A_n\}$})
	$(g,K) \stackrel{\{A_1,\ldots,A_n\}}{\leadsto} (g_1,K_1)$
	is the following procedure:
	
	$\bullet$
	Let $K_0$ be the space obtained by cutting out $A_1,\ldots,A_n$ from $K$
	(i.e. assume $\{P_1,\ldots,P_k\} = \{the$ $components$ $of$ $K \setminus (A_1 \cup A_2 \cup \ldots \cup A_n)\}$,
	and let $K_0 = \coprod_{i=1}^{k} \overline{P_i}$).
	Let $g_0: K_0 \to \mathbb{R}^{3}$ be the map induced by $g$.
	For each $\alpha$ a $2$-cell of $X(A_1,A_2,\ldots,A_n)$,
	there exists exactly two components of $g_{0}^{-1}(\overline{\alpha})$ lying in the boundary of $K_0$.
	We denote them by $D_{\alpha}^{+},D_{\alpha}^{-}$.
	Let $h$ be the equivalence relation such that 
	$x \stackrel{h}{\sim} y$ if there exists $\alpha$ a $2$-cell of $X(A_1,A_2,\ldots,A_n)$,
	$x \in D_{\alpha}^{+}$, 
	$y \in D_{\alpha}^{-}$,
	and $g_0(x) = g_0(y)$.
	Let $K_1$ be the identification space $K_0 / \sim_h$.
	Assume $h_*: K_0 \to K_1$ is the identification map induced by $h$.
	Let $g_1: K_1 \to \mathbb{R}^{3}$ be the map given by following commutative diagram.
	\begin{center}
		$\xymatrix{
			& K_0 \ar[d]^{h_*} \ar[r]_{g_0}
			& \mathbb{R}^{3}  \ar[d]_{id}       \\
			& K_1 \ar[r]^{g_1}   & \mathbb{R}^{3}                }$
	\end{center}

    Hence the cancellation operation $(g,K) \stackrel{\{A_1,\ldots,A_n\}}{\leadsto} (g_1,K_1)$ has been defined
    ($K_1$ is a branched $3$-manifold and
    $g_1: K_1 \to \mathbb{R}^{3}$ is a branched immersion).
    
	(ii)
	For each $x \in X(A_1,A_2,\ldots,A_n)$,
	we say the cancellation $(g,K) \stackrel{\{A_1,\ldots,A_n\}}{\leadsto} (g_1,K_1)$
	is \emph{regular at} $x$ if $\#(h_*(\partial K_0 \cap g^{-1}(x))) = 1$ (where $h_*: K_0 \to K_1$ is the identification map of cancellation).
	The cancellation $(g,K) \stackrel{\{A_1,\ldots,A_n\}}{\leadsto} (g_1,K_1)$ is called \emph{regular} if it is regular at every $x \in X(A_1,A_2,\ldots,A_n)$.
	
	(iii)
	Assume that the cancellation $(g,K) \stackrel{\{A_1,\ldots,A_n\}}{\leadsto} (g_1,K_1)$ is regular.
	Let $T: X(A_1,$ $\ldots,A_n) \to K_1$ be the map
	sending each $x \in X(A_1,A_2,\ldots,A_n)$ to $h_*(\partial K_0 \cap g^{-1}(x))$
	(then $X(A_1,A_2,\ldots,A_n)$ is homeomorphically embedded into $K_1$ by $T$).
	We call $T$ the \emph{associated map} of the cancellation of $\{A_1,\ldots,A_n\}$.
\end{defn}

\begin{remark}\label{regular means}\rm
    We give some remarks on cancellations.
    For $(g,K) \stackrel{\{A_1,\ldots,A_n\}}{\leadsto} (g_1,K_1)$ the cancellation of $A_1,\ldots,A_n$, the following hold:
    
    (a)
    $g_1(\partial K_1) = \overline{g(\partial K) \setminus \partial R(g)} = \bigcup_{i=1}^{n-1}D_i(g \mid_{\partial K})$ 
    (where $n = \max_{x \in \mathbb{R}^{3} \setminus g(\partial K)} \omega(g \mid_{\partial K},x)$).
    
    (b)
    $G(g) = X(A_1,\ldots,A_n) \cap \partial R(g_1)$.
    $X(A_1,\ldots,A_n)$ is a good $2$-complex in $R(g_1)$ with respect to $G(g)$.
    
    (c)
    The cancellation $(g,K) \stackrel{\{A_1,\ldots,A_n\}}{\leadsto} (g_1,K_1)$ is regular at every $x \in X(A_1,A_2,\ldots,A_n) \setminus S(X(A_1,A_2,\ldots,A_n))$.
    So $(g,K) \stackrel{\{A_1,\ldots,A_n\}}{\leadsto} (g_1,K_1)$ is regular if and only if it is regular at every $x \in S(X(A_1,A_2,\ldots,A_n))$.
	
	(d)
	If the cancellation $(g,K) \stackrel{\{A_1,\ldots,A_n\}}{\leadsto} (g_1,K_1)$ is regular,
	then $T(G(g)) \subseteq \partial K_1$ (where $T$ is the associated map of the cancellation).
\end{remark}

\section{Embedded $2$-complexes in $3$-manifolds}\label{section 3}\

In this section,
we introduce some embedded $2$-complexes in $3$-manifolds,
and give the steps to construct them.
We will let them be the associated $2$-complexes of cancellable domains 
to yield cancellable domains in Section \ref{section 5}.

\subsection{$(M,G)$-simple $2$-complex}

\begin{figure}\label{points in simple complex}
	\includegraphics[width=0.29\textwidth]{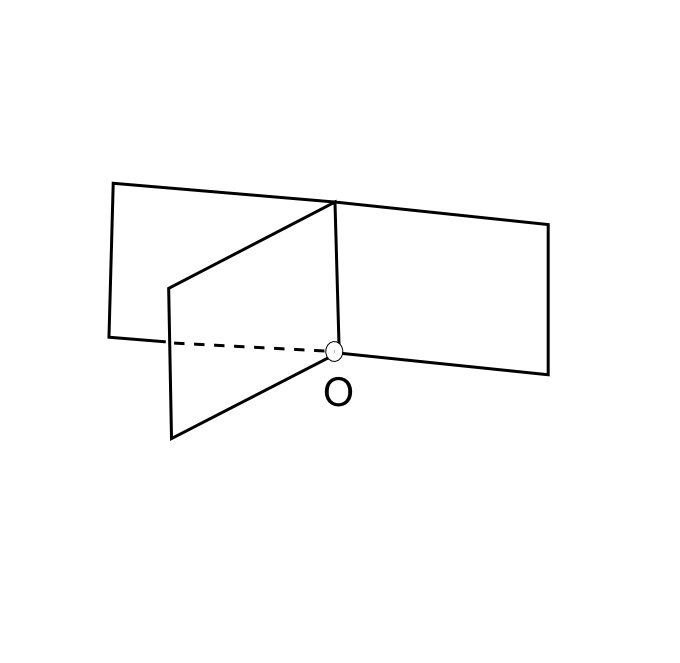}
	(a)
	\includegraphics[width=0.29\textwidth]{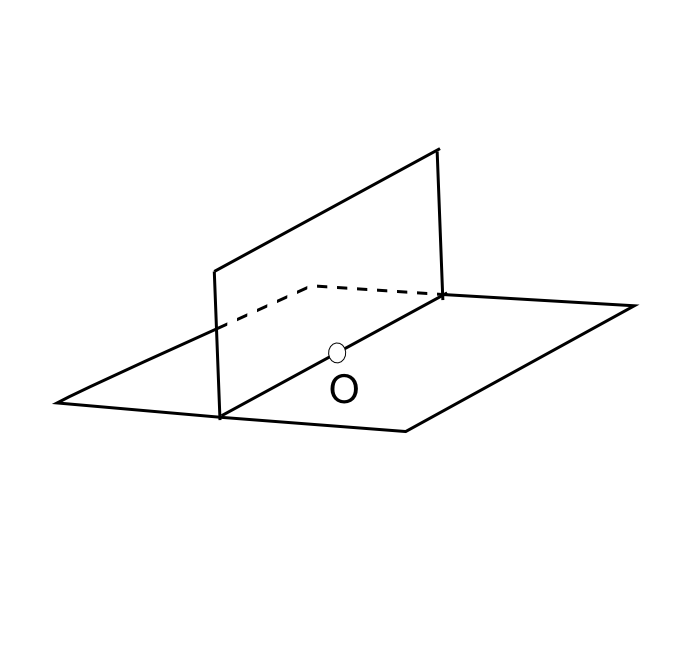}
	(b)
	\includegraphics[width=0.29\textwidth]{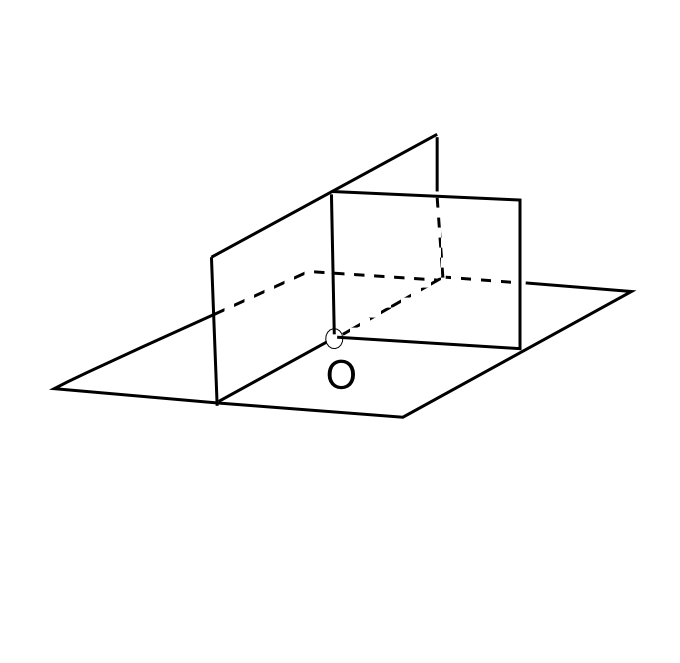}
	(c)
	\caption{The points in a $(M,G)$-simple $2$-complex.}
\end{figure}

\begin{defn}\label{relative}\rm
	Let $M$ be a compact $3$-manifold with nonempty boundary and 
	$G \subseteq \partial M$ an embedded graph such that all vertices have degree $2$ or $3$.
	Let $X \subseteq M$ be $2$-complex in $M$.
	
	(i)
	We say $X$ is a \emph{$(M,G)$-simple $2$-complex} if:
	
	$\bullet$
	$X \cap \partial M = G$.
	
	$\bullet$
	For each $t \in G \setminus \{v \in V(G) \mid deg_G(v) = 3\}$,
	there exists an open neighborhood $N(t)$ of $t$ in $M$ such that $(N(t) \cap X,t)$ 
	is homeomorphic to $(\mathbb{R}^{2}_{+},0)$.
	
	$\bullet$
	For each $t \in \{v \in V(G) \mid deg_G(v) = 3\}$,
	there exists an open neighborhood $N(t)$ of $t$ in $M$ such that $(N(t) \cap X,t)$
	is homeomorphic to 
	$(\{x = 0, y \geqslant 0, z \geqslant 0\} \cup \{y = 0, z \geqslant 0\},0)$ (see Figure \ref{points in simple complex} (a)).
	(where $\{x = 0, y \geqslant 0, z \geqslant 0\} \cup \{y = 0, z \geqslant 0\}$ denotes a subset of $\mathbb{R}^{3}$, $(x,y,z)$ is the coordinates of $\mathbb{R}^{3}$, and the following is same)
	
	$\bullet$
	For each $t \in X \setminus G$,
	there exists an open neighborhood $N(t)$ of $t$ in $M$ such that $(N(t) \cap X,t)$ is homeomorphic to one of (a) $\sim$ (c):
	
	(a)
	$(\mathbb{R}^{2},0)$.
	
	(b)
	$(\{z = 0\} \cup \{x = 0, z \geqslant 0\},0)$ (see Figure \ref{points in simple complex} (b)).
	
	(c)
	$(\{z = 0\} \cup \{x = 0, z \geqslant 0\} \cup \{y = 0, x \geqslant 0, z \geqslant 0\},0)$ (see Figure \ref{points in simple complex} (c)).
	
	$\bullet$
	$\#(\{the$ $components$ $of$ $\partial M \setminus G\}) = \#(\{the$ $components$ $of$ $M \setminus X\})$,
	and each component of $M \setminus X$ contains exactly one component of $\partial M \setminus G$.
	
	$\bullet$
	Assume $\{A_1,A_2,\ldots,A_n\} = \{the$ $components$ $of$ $\partial M \setminus G\}$,
	$\{B_1,B_2,\ldots,B_n\} = \{the$ $components$ $of$ $M \setminus X\}$,
	and $A_k \subseteq B_k$
	($\forall k \in \{1,2,\ldots,n\}$).
	Choose $x_k \in A_k$,
	assume $i_*: \pi_1(A_k,x_k) \to \pi_1(M,x_k),
	j_*: \pi_1(B_k,x_k) \to \pi_1(M,x_k)$
	are the maps induced by
	the inclusion maps of $A_k,B_k$ into $M$.
	Then $i_*(\pi_1(A_k,x_k)) = j_*(\pi_1(B_k,x_k))$ ($\forall k \in \{1,2,\ldots,n\}$).
	
	(ii)
	If $X$ is a $(M,G)$-simple $2$-complex,
	we say $Y$ is a \emph{good subcomplex} of $X$ if:
	$Y$ is a subcomplex of $X$ such that 
	$Y$ is a good complex in $M$ with respect to $G$.
	We denote by $sub(X)$ the set consisting of all good subcomplex of $X$.
\end{defn}

Obviously,
a $(M,G)$-simple $2$-complex is a good $2$-complex in $M$ with respect to $G$.
Moreover,
given a covering space $p: (\tilde{M},\tilde{x}_k) \to (M,x_k)$ ($x_k \in A_k$),
then the inclusion map $i_B: (B_k,x_k) \to (M,x_k)$ has a lift $\tilde{i}_B: (B_k,x_k) \to (\tilde{M},\tilde{x}_k)$ if and only if the inclusion map $i_A: (A_k,x_k) \to (M,x_k)$ has a lift $\tilde{i}_A: (A_k,x_k) \to (\tilde{M},\tilde{x}_k)$.

\begin{defn}\label{appropriate}\rm
	Let $M$ be a compact $3$-manifold with nonempty boundary and 
	$G \subseteq \partial M$ an embedded graph such that all vertices have degree $2$ or $3$.
	$(M,G)$ is called \emph{appropriate} if:
	for each $e \in E(G)$ such that both $2$ sides of $e$ lie in the same component $A$ of $\partial M \setminus G$ (i.e. $Int(e) \subseteq Int(\overline{A})$),
	all immersed loops in $\overline{A}$ that intersect with $e$ one time transversely are not null-homotopic in $M$.
\end{defn}

\begin{figure}\label{not appropriate}
\begin{center}
	\includegraphics[width=0.5\textwidth]{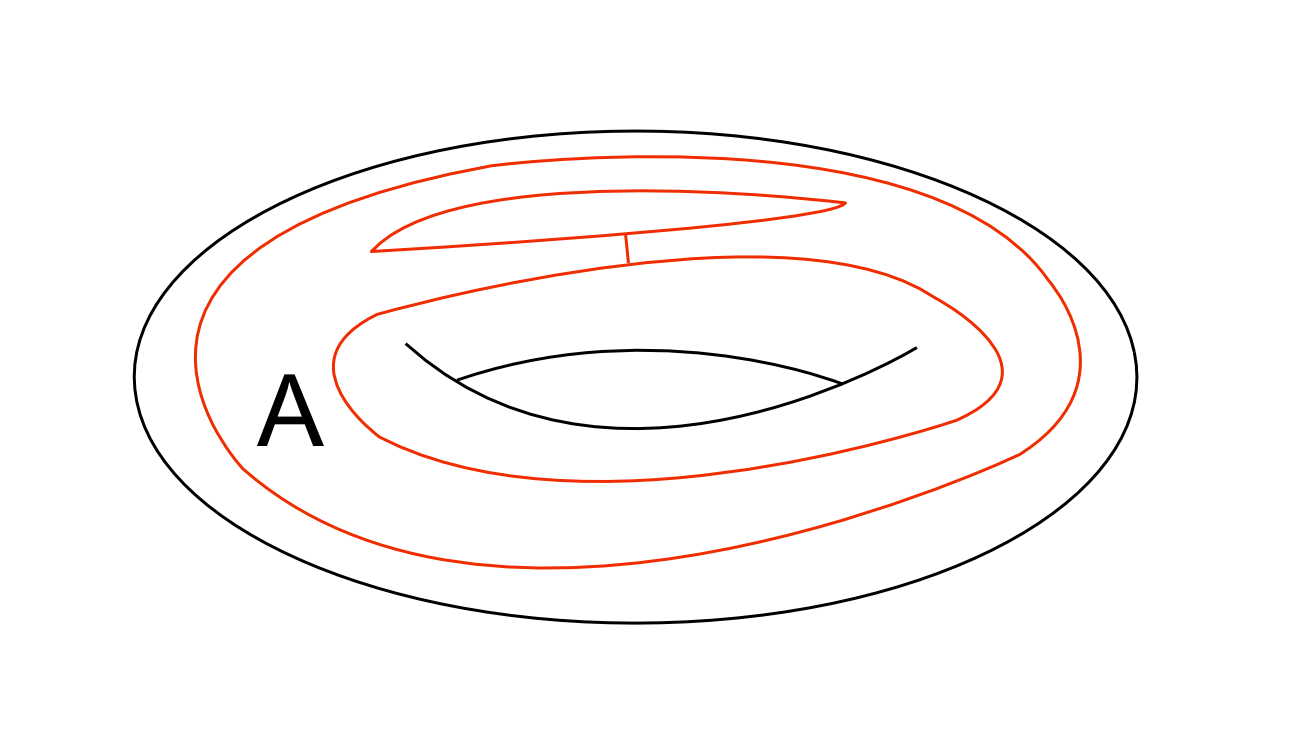}
\end{center}
	\caption{$M = S^{1} \times D^{2}$, $A$ is a component of $\partial M \setminus G$. Then $(M,G)$ is not appropriate.}
\end{figure}

Figure \ref{not appropriate} gives an example of $(M,G)$ to be not appropriate.
And
we can state Definition \ref{appropriate} in a different way.
$(M,G)$ is appropriate if:
for each $e \in E(G)$ such that both $2$ sides of $e$ lie in the same component $A$ of $\partial M \setminus G$.
Choose $x \in A$,
assume $i_*: \pi_1(A,x) \to \pi_1(M,x),i_{*}^{'}: \pi_1(A \cup Int(e),x) \to \pi_1(M,x)$
are the maps induced by the inclusion maps of $A,A \cup Int(e)$ into $M$,
then $i_*(\pi_1(A,x)) \ne i^{'}_{*}(\pi_1(A \cup Int(e),x))$.

\begin{lm}\label{appropriate = lifted}
	Let $M$ be a compact $3$-manifold with nonempty boundary and 
	$G \subseteq \partial M$ an embedded graph such that all vertices have degree $2$ or $3$.
	$(M,G)$ is appropriate if and only if:
	for each component $A$ of $\partial M \setminus G$ and $x \in A$,
	there exists a covering space $p: (\tilde{M},\tilde{x}) \to (M,x)$ and $A_0 \subseteq \tilde{M}$ a closed domain
	such that $\tilde{x} \in A_0$,
	and $p \mid_{Int(A_0)}$ is a homeomorphism between $Int(A_0)$ and $A$.
\end{lm}

\begin{proof}
	(i)
	We assume that for each $A$ a component of $\partial M \setminus G$ and $x \in A$,
	there exists
	$p: (\tilde{M},\tilde{x}) \to (M,x)$ a covering space and $A_0 \subseteq \tilde{M}$ a closed domain
	such that $\tilde{x} \in A_0$,
	and $p \mid_{Int(A_0)}$ is a homeomorphism between $Int(A_0)$ and $A$.
	We will prove that $(M,G)$ is appropriate.
	
	The inclusion map $i: (A,x) \to (M,x)$ has a lift $\tilde{i}: (A,x) \to (\tilde{M},\tilde{x})$ such that $\tilde{i}(A) = Int(A_0)$.
	But $p^{-1}(e) \cap A_0$ has $2$ different components for every $e \in E(G)$ such that both $2$ sides of $e$ lie in $A$.
	Hence $i_*(\pi_1(A,x)) \subseteq p_*(\pi_1(\tilde{M},\tilde{x}))$,
	and $i^{'}_{*}(\pi_1(A \cup Int(e),x)) \nsubseteq p_*(\pi_1(\tilde{M},\tilde{x}))$
	(where $p_*: \pi_1(\tilde{M},\tilde{x}) \to \pi_1(M,x),
	i_*: \pi_1(A,x) \to \pi_1(M,x),
	i_{*}^{'}: \pi_1(A \cup Int(e),x) \to \pi_1(M,x)$
	are the maps induced by $p$, $i$, and the inclusion map of $A \cup Int(e)$ into $M$).
	So $i_*(\pi_1(A,x)) \ne i^{'}_{*}(\pi_1(A \cup Int(e),x))$.
	
	Hence $(M,G)$ is appropriate.
	
	(ii)
	Assume $(M,G)$ is appropriate.
	For each component $A$ of $\partial M \setminus G$ and $x \in A$,
	let $p: (\tilde{M},\tilde{x}) \to (M,x)$ be the covering space such that 
	$p_*(\pi_1(\tilde{M},\tilde{x})) = i_*(\pi_1(A,x))$ ($p_*,i_*$ are same as (i)).
	
	Let $\tilde{i}: (A,x) \to (\tilde{M},\tilde{x})$ be a lift of the inclusion map $i: (A,x) \to (M,x)$.
	Let $A_0$ be the closure of $\tilde{i}(A)$.
	For each edge $e \in E(G)$ such that both $2$ sides of $e$ lie in $A$,
	$i^{'}_{*}(\pi_1(A \cup Int(e),x)) \nsubseteq p_*(\pi_1(\tilde{M},\tilde{x}))$ (where $i_{*}^{'}: \pi_1(A \cup Int(e),x) \to \pi_1(M,x)$ is the map induced by the inclusion map of $A \cup Int(e)$ into $M$).
	Then $p^{-1}(e) \cap A_0$ has $2$ different components.
	Hence $Int(A_0) = \tilde{i}(A)$.
\end{proof}

\begin{lm}\label{case 1 appropriate}
	(i) If $M$ is a compact $3$-manifold with nonempty boundary and
	$g: M \to \mathbb{R}^{3}$ is an immersion,
	then $(R(g),G(g))$ is appropriate.
	
	(ii) If $X$ is a $(R(g),G(g))$-simple $2$-complex,
	then there exists $A_1,\ldots,A_n$ cancellable domains such that $X$ is their associated $2$-complex.
	And the cancellation of $\{A_1,\ldots,A_n\}$ is regular.
\end{lm}

\begin{proof}
	(i)
	For each component $A$ of $\partial R(g) \setminus G(g)$,
	let $S$ be $\overline{(g \mid_{\partial M})^{-1}(A)}$.
	Then $g \mid_{Int(S)}$ is a homeomorphism between $Int(S)$ and $A$.
	So $(R(g),G(g))$ is appropriate (by Lemma \ref{appropriate = lifted}).
	
	(ii)
	For each component $A$ of $\partial R(g) \setminus G(g)$,
	assume $B$ is the component of $R(g) \setminus X$ containing $A$,
	and $x \in A$.
	Then $i_{A*}(\pi_1(A,x)) = i_{B*}(\pi_1(B,x))$,
	where $i_{A*}: \pi_1(A,x) \to \pi_1(R(g),x), i_{B*}: \pi_1(B,x) \to \pi_1(R(g),x)$ are the maps induced by the inclusion maps of $A,B$ into $R(g)$.
	So there exists $\tilde{i}_{B}: (B,x) \to (M,g^{-1}(x) \cap \partial M)$ a lift of $i_B: (B,x) \to (R(g),x)$ (the inclusion map of $B$ into $R(g)$),
	and $\tilde{i}_B(B)$ contains $g^{-1}(A) \cap \partial M$.
	So there exist $A_1,\ldots,A_n$ cancellable domains such that $X$ is their associated $2$-complex.
	And we can verify that the cancellation of $\{A_1,\ldots,A_n\}$ is regular,
	since every point $t \in S(X)$ has a neighborhood $N(t)$ such that $(N(t) \cap X,t)$ is homeomorphic to one of Figure \ref{points in simple complex} (a), (b), (c).
\end{proof}

\begin{defn}\rm
	Let $M$ be a compact $3$-manifold with nonempty boundary and 
	$G_0 \subseteq \partial M$ an embedded graph such that all vertices have degree $2$ or $3$.
	Let $X_0 \subseteq M$ be a good $2$-complex in $M$ with respect to $G_0$.
	Assume $N$ is a subgraph of $S(X_0)$ and $G \subseteq \partial M$ is a thin trivalent graph (Definition \ref{thin trivalent graph}) of $G_0$ in $\partial M$.
    
	(i)
	Let $M_1$, $M_2$, \ldots, $M_s$
	be the components obtained by
	cutting off $X_0$ from $M$ (``cut off'' means to delete the set from the space and do a path compactification),
	and $i_k: M_k \to M$ ($\forall k \in \{1,2,\ldots,m\}$) is continuous map induced by the ``cutting off'' ($i_k \mid_{Int(M_k)}$ is an inclusion map).
	Let $G_k = \{x \in \partial M_k \mid i_k(x) \in G \cup N\}$.
	$(M,X_0,G \cup N)$ is called \emph{appropriate} if:
	for each $k \in \{1,2,\ldots,m\}$,
	$G_k$ is an embedded graph such that all vertices have degree $2$ or $3$,
	and $(M_k,G_k)$ is appropriate.
	
	(ii)
	Assume $(M,X_0,G \cup N)$ is appropriate.
	If $X_k \subseteq M_k$ is a $(M_k,G_k)$-simple $2$-complex ($\forall k \in \{1,2,\ldots,s\}$),
	we say the $2$-complex $X = \bigcup_{k=1}^{s} i_k(X_k)$ is a \emph{$(M,X_0,G \cup N)$-simple $2$-complex}.
	In addition, for all $k \in \{1,2,\ldots,s\}$ and $X^{'}_{k} \in sub(X_k)$,
	we say $\bigcup_{k=1}^{s} i_k(X^{'}_{k})$ is an \emph{$X_0$-good subcomplex} of $X$,
	and we denote by $sub_{X_0}(X)$ the set consisting of all $X_0$-good subcomplexes of $X$.
\end{defn}

\subsection{The construction of the $(M,G)$-simple $2$-complex}\

\begin{prop}\label{exists-regular}
	Let $M$ be a compact $3$-manifold with nonempty boundary and 
	$G \subseteq \partial M$ an embedded graph such that all vertices have degree $2$ or $3$.
	Assume $(M,G)$ is appropriate,
	then there exists a $(M,G)$-simple $2$-complex.
\end{prop}

\begin{proof}
	Assume $A_1,\ldots,A_n$ are the components of $\partial M \setminus G$,
	$x_k \in A_k$ ($k \in \{1,2,\ldots,n\}$).
	Let $p_k: (\tilde{M}_k,\tilde{x}_k) \to (M,x_k)$ be a covering space such that
	$p_{k*}(\pi_1(\tilde{M}_k,\tilde{x}_k)) = i_{k*}(\pi_1(A_k,x_k))$,
	where $p_{k*}: \pi_1(\tilde{M}_k,\tilde{x}_k) \to \pi_1(M,x_k)$, $i_{k*}: \pi_1(A_k,x_k) \to \pi_1(M,x_k)$ are induced by $p_k$ and the inclusion map of $A_k$ into $M$.   
	Then there exists a closed domain $S_k \subseteq \partial \tilde{M}_k$ such that $p_k \mid_{Int(S_k)}$ is a homeomorphism between $Int(S_k)$ and $A_k$. 
	
	Assume $p: \coprod_{k=1}^{n} \tilde{M}_k \to M$ is the map such that $p \mid_{\tilde{M}_k} = p_k$ ($\forall k \in \{1,2,\ldots,n\}$).
	There exists $M_0 \subseteq \coprod_{k=1}^{n} \tilde{M}_k$ such that:
	($\forall k \in \{1,2,\ldots,n\}$)
	assume $L_k = M_0 \cap \tilde{M}_k \subseteq \tilde{M}_k$,
	then $L_k$ is connected,
	$S_k \subseteq L_k$,
	$Int(L_1),\ldots,Int(L_n)$ are homeomorphically embedded into $M$ by $p$,
	and there exists $X(M_0)$ a good $2$-complex in $M$ with respect to $G$
	such that $\{p(Int(L_1)),\ldots,p(Int(L_n))\} = \{the$ $components$ $of$ $Int(M) \setminus X(M_0)\}$
	(i.e. $p$ maps $L_1,\ldots,L_k$ homeomorphically to the closed components obtained by cutting off $X(M_0)$ from $M$).
	Then $M_0$ is closed,
	$p(M_0) = M$,
	$p(M_0 \setminus U) \ne M$ for any open set $U \subseteq M_0$.
	
	Moreover,
	$X(M_0) = \overline{p(\partial M_0) \setminus \partial M}$.
	Assume $p_0: M_0 \to M$ is the map such that $p_0 = p \mid_{M_0}$.
	Then $X(M_0) = \{x \in M \mid \#(p_{0}^{-1}(x)) \geqslant 2\}$,
	and the embedded graph $S(X(M_0)) = \{x \in M \mid \#(p_{0}^{-1}(x)) \geqslant 3\}$.
	
	\begin{figure}\label{thicken e}
		\begin{center}
			\includegraphics[width=0.6\textwidth]{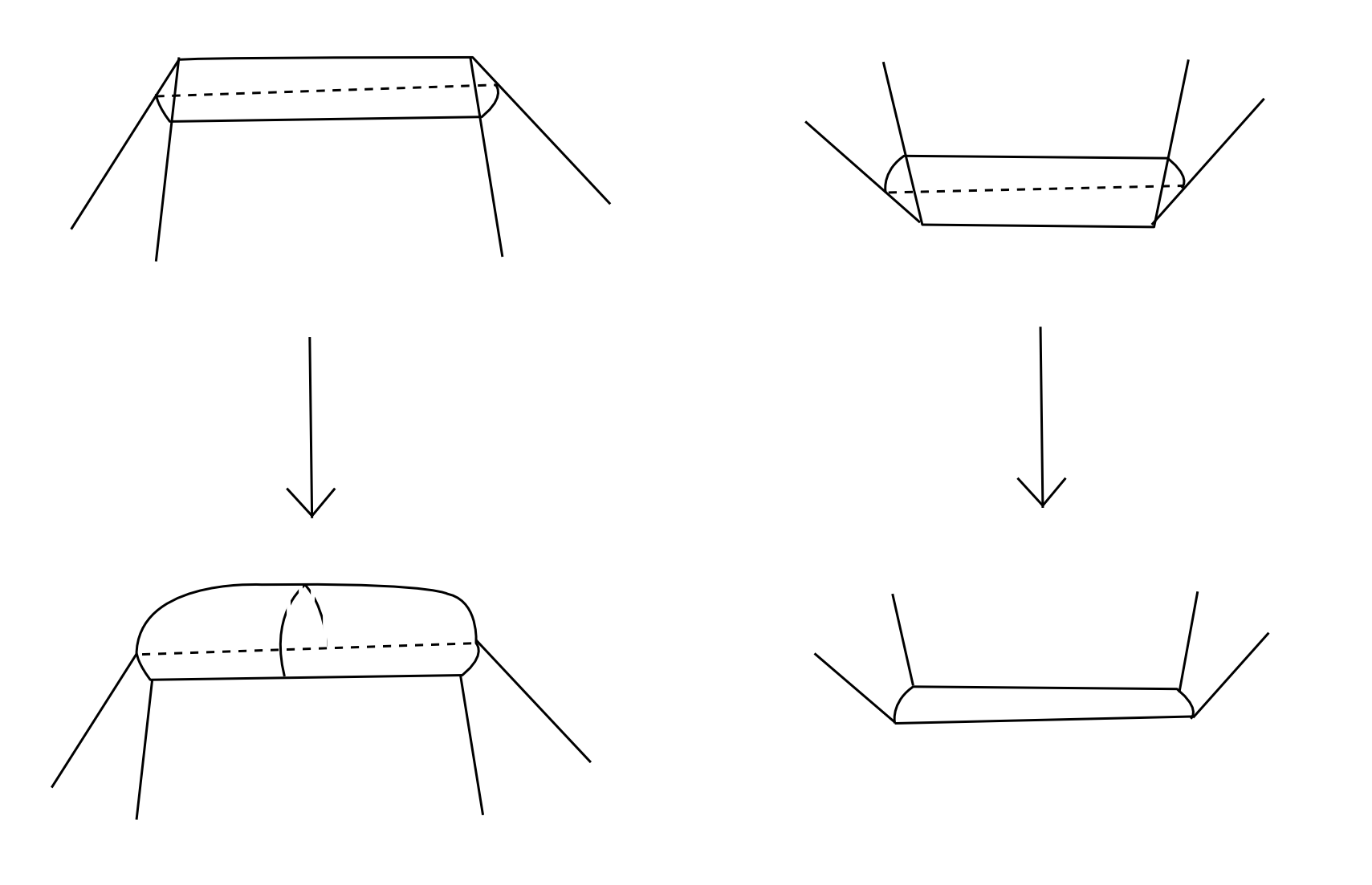}
		\end{center}
		\caption{Thickening an edge.}
	\end{figure}
	
	In the following,
	we adjust $M_0$ step by step (after each step, $X(M_0)$ is also a good $2$-complex in $M$ with respect to $G$).
	$X(M_0)$ will be a $(M,G)$-simple $2$-complex after all steps finished.
	
	(a)
	(Thicken an edge)
	If there exists an edge $e \in E(S(X(M_0)))$,
	$\#(p_{0}^{-1}(x)) \geqslant 4$ for an $x \in Int(e)$,
	we adjust $M_0$ by the process of thicken $e$ (Picture \ref{thicken e}):
	
	$\bullet$
	Choose $N(e)$ an arbitrarily small open regular neighborhood of $e$ in $M$,
	and choose $e_0$ a component of $p_{0}^{-1}(e)$.
	Assume $N(e_0)$ is the component of $p_{0}^{-1}(N(e))$ containing $e_0$.
	Let 
	\begin{center}
	$M^{'}_{0} = \overline{(M_0 \setminus p^{-1}_{0}(N(e))) \cup N(e_0)}$
	\end{center}
    and replace $M_0$ by $M^{'}_{0}$.
	
	Obviously,
	$\#(\{e \in E(S(X(M_0))) \mid \exists x \in Int(e),\#(p_{0}^{-1}(x)) \geqslant 4\})$
	reduces after thickening an edge.
	
	\begin{figure}\label{thicken v}
		\begin{center}
			\includegraphics[width=0.5\textwidth]{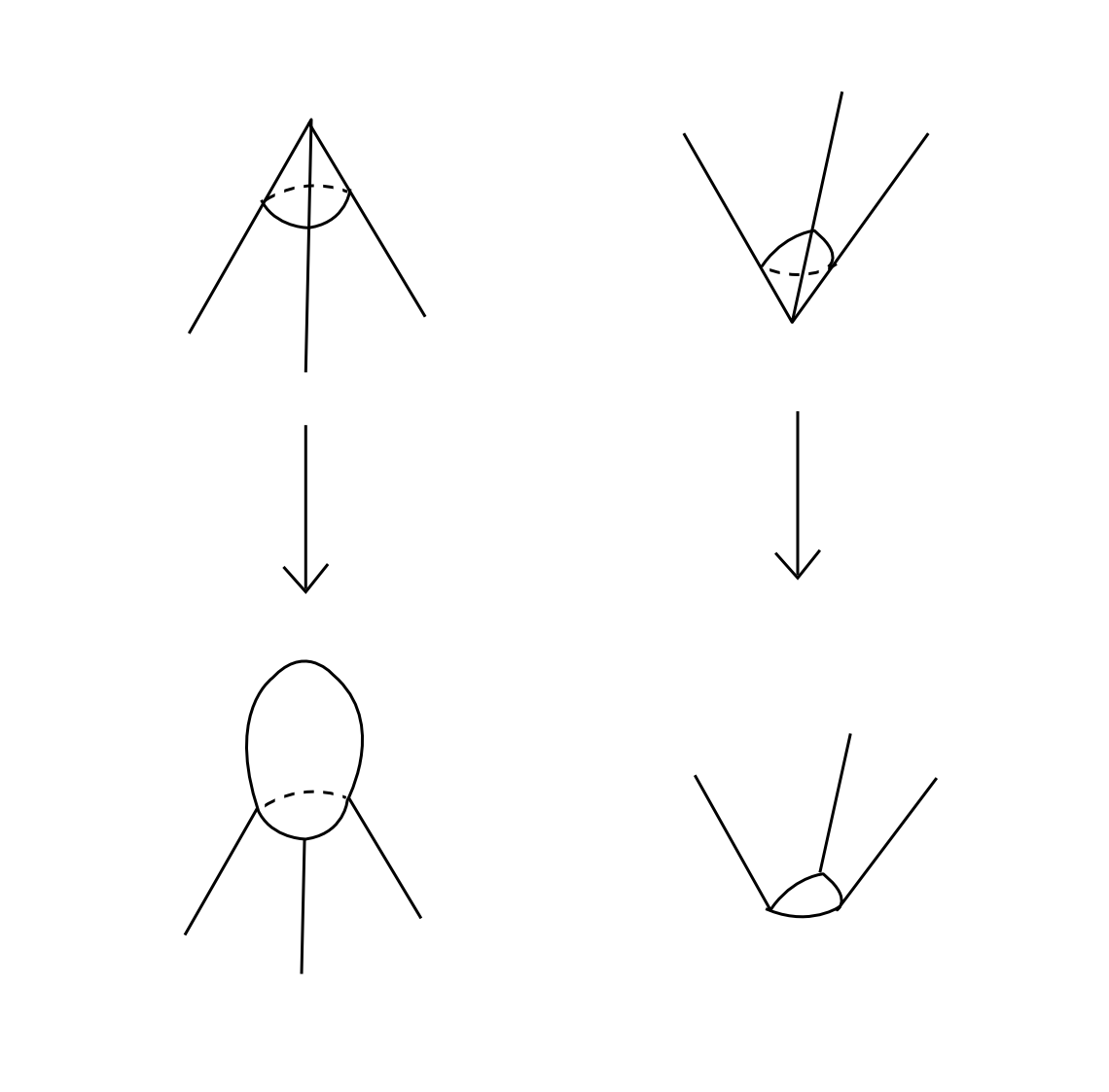}
		\end{center}
		\caption{Thickening a vertex.}
	\end{figure}
	
	(b)
	(Thicken a vertex)
	After all above thickenings (of edges),
	$M$ satisfies that
	for all $e \in E(S(X(M_0)))$ and $x \in Int(e)$,
	$\#(p_{0}^{-1}(x)) = 3$.
	
	If there exists $v \in V(S(X(M_0)))$
	such that $\#(p_{0}^{-1}(v)) \geqslant 5$,
	we adjust $M_0$ by the process of thicken $v$ (Picture \ref{thicken v}):
	
	$\bullet$
	Choose $N(v)$ an arbitrarily small open regular neighborhood of $v$ in $M$,
	and choose $v_0 \in p_{0}^{-1}(v)$.
	Assume $N(v_0)$ is the component of $p_{0}^{-1}(N(v))$ containing $v_0$.
	Let 
	\begin{center}
	$M^{'}_{0} = \overline{(M_0 \setminus p_{0}^{-1}(N(v))) \cup N(v_0)}$
	\end{center}
and replace $M_0$ by $M^{'}_{0}$.

	Obviously,
	the number of $v \in V(S(X(M_0)))$ such that $\#(p_{0}^{-1}(v)) \geqslant 5$ reduces after thickening a vertex.
	The edges produced in thickening a vertex satisfy that for each $x$ in their interior, $\#(p_{0}^{-1}(x)) = 3$.
	
	After all above thickenings (of vertices),
	$M$ satisfies that
	$\#(p_{0}^{-1}(v)) = 4$ for each $v \in \{v \in V(S(X(M_0))) \mid deg_{S(X(M_0))}(v) > 2\}$.
	We denote $\{v \in V(S(X(M_0))) \mid deg_{S(X(M_0))}(v) > 2\} = \{v \in V(S(X(M_0))) \mid deg_{S(X(M_0))}(v) = 4\}$ by $T(X(M_0))$ in the following.
	For each $t \in S(X(M_0)) \setminus (T(X(M_0)) \cup \{v \in V(G) \mid deg_G(v) = 3\})$,
	(then $t$ is either in the interior of an edge of $S(X(M_0))$ or a vertex of $S(X(M_0))$ with degree $2$)
	there exists $N(t)$ an open neighborhood of $t$ in $M$ such that $(N(t) \cap X(M_0),t)$ is homeomorphic to $(\{z = 0\} \cup \{x = 0, z \geqslant 0\},0)$.
	For each $t \in T(X(M_0))$,
	there exists $N(t)$ an open neighborhood of $t$ in $M$ such that $(N(t) \cap X(M_0),t)$ is homeomorphic to
	$(\{z = 0\} \cup \{x = 0, z \geqslant 0\} \cup \{y = 0, x \geqslant 0, z \geqslant 0\},0)$.
	For each $t \in \{v \in V(G) \mid deg_G(v) = 3\}$,
	there exists an open neighborhood $N(t)$ of $t$ in $M$ such that $(N(t) \cap X,t)$
	is homeomorphic to 
	$(\{x = 0, y \geqslant 0, z \geqslant 0\} \cup \{y = 0, z \geqslant 0\},0)$ (since $X(M_0)$ is a good $2$-complex in $M$ with respect to $G$).
	
	Obviously,
	$X(M_0)$ is a $(M,G)$-simple $2$-complex.
\end{proof}

\begin{cor}\label{exists-thin regular}
Let $M$ be a compact $3$-manifold with nonempty boundary and 
$G_0 \subseteq \partial M$ an embedded graph such that all vertices have degree $2$ or $3$.
Let $X_0 \subseteq M$ be a good $2$-complex in $M$ with respect to $G_0$.
Assume $N$ is a subgraph of $S(X_0)$ and $G \subseteq \partial M$ is a thin trivalent graph of $G_0$ in $\partial M$.
If $(M,X_0,G \cup N)$ is appropriate,
then there exists a $(M,X_0,G \cup N)$-simple $2$-complex.
\end{cor}

\begin{remark}\rm
	In Proposition \ref{exists-regular},
	We prove that there exists a $(M,G)$-simple $2$-complex if $(M,G)$ is appropriate.
	Auctually,
	a $(M,G)$-simple $2$-complex can be constructed through the proof of Proposition \ref{exists-regular}.
    Similarly,
    we can construct a $(M,X_0,G \cup N)$-simple $2$-complex in Corollary \ref{exists-thin regular}.
    In the rest of this paper,
    we will always assume that a $(M,G)$-simple $2$-complex can be constructed immediately when we know $(M,G)$ is appropriate,
    and assume a $(M,X_0,G \cup N)$-simple $2$-complex can be constructed immediately when we know $(M,X_0,G \cup N)$ is appropriate.
\end{remark}

\section{Inscribed set}\label{section 4}\

\begin{defn}[Inscribed set]\label{inscribed set}\rm 
	For a closed oriented surface $\Sigma$,
	let $f: \Sigma \to \mathbb{R}^{3}$ be an immersion.
	Assume $n = \max_{x \in \mathbb{R}^{3} \setminus f(\Sigma)}\omega(f,x)$.
	The following process induces decreasingly on $k$, until $k = 1$.
	
	For step $1$:
	If $(D_n(f),G_n(f))$ is appropriate,
	then there exists $\tilde{X}_n$ a $(D_n(f),G_n(f))$-simple $2$ complex.
	Let $\zeta_n = \{(\tilde{X}_n,X_n) \mid X_n \in sub(\tilde{X}_n)\}$.
	If $(D_n(f),G_n(f))$ is not appropriate,
	then $\zeta_n = \emptyset$.
	
	For step $n-k+1$ ($1 \leqslant k \leqslant n-1$):
	Assume $\zeta_{k+1}$ is obtained in the step $n-k$.
    $\zeta_k$ is obtained as follows:
	
	For each $A = \{(\tilde{X}_{k+1},X_{k+1}), \ldots, (\tilde{X}_n,X_n)\} \in \zeta_{k+1}$,
	assume $N = \overline{S(X_{k+1}) \setminus S(X_{k+2})}$.
	And we define $Q(A)$ by the following rules:
	
	$\bullet$
	If $(D_k(f),X_{k+1},N \cup G_k(f))$ is appropriate,
	choose $\tilde{X}_k$ a $(D_k(f),X_{k+1},N \cup G_k(f))$-simple $2$-complex.
	Let
	$Q(A) =
	\{(\tilde{X}_k,X_k) \mid X_k \in sub_{X_{k+1}}(\tilde{X}_k)\}$.
	
	$\bullet$
	If $(D_k(f),X_{k+1},N \cup G_k(f))$ is not appropriate,
	then $Q(A) = \emptyset$.
	
	Let $\zeta_k = \bigcup_{A \in \zeta_{k+1}, Q(A) \ne \emptyset} \bigcup_{B \in Q(A)} (A \cup B)$. ($\zeta_k = \emptyset$ if $\zeta_{k+1} = \emptyset$)
	
	In the end,
	we obtain an \emph{inscribed set} $\zeta = \zeta_1$,
	and we obtain the sets $\zeta_2,\ldots,\zeta_n$ ($\zeta_k = \{(\tilde{X}_k,X_k), \ldots, (\tilde{X}_n,X_n)\}$) through the process (call $\zeta_k$ the \emph{$k$th-inscribed set} of $\zeta$).
\end{defn}

\begin{defn}\label{good complex set}\rm
	For a closed oriented surface $\Sigma$,
	let $f: \Sigma \to \mathbb{R}^{3}$ be an immersion.
	Assume $n = \max_{x \in \mathbb{R}^{3} \setminus f(\Sigma)}\omega(f,x)$.
	Let $\zeta$ be an inscribed set of $f$.
	An element $\{(\tilde{X}_1,X_1),\ldots,(\tilde{X}_n,X_n)\} \in \zeta$ is \emph{good}
	if $X_1 = \emptyset$.
	We denote $\{A \in \zeta \mid A$ $is$ $good\}$ by $I(\zeta)$.
\end{defn}

\begin{figure}\label{N}
		\includegraphics[width=0.29\textwidth]{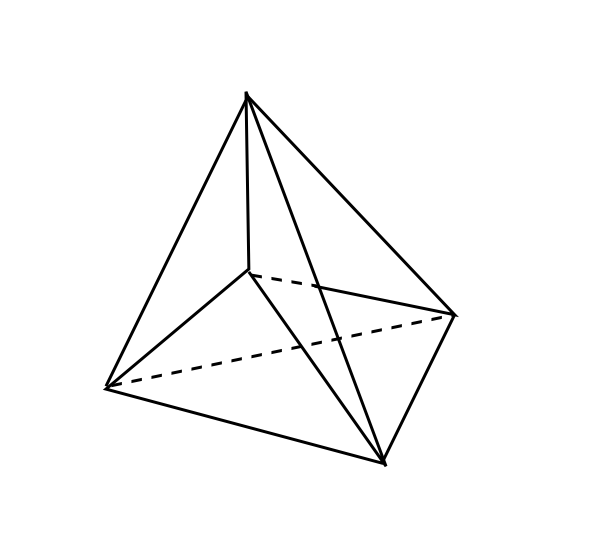}
		(a)
		\includegraphics[width=0.29\textwidth]{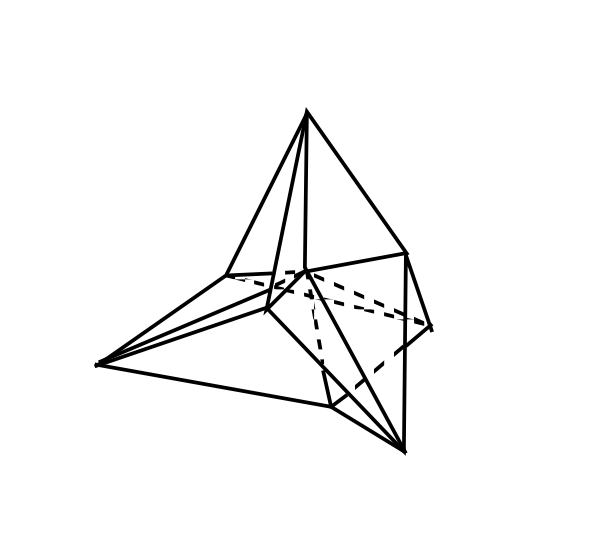}
		(b)
		\includegraphics[width=0.29\textwidth]{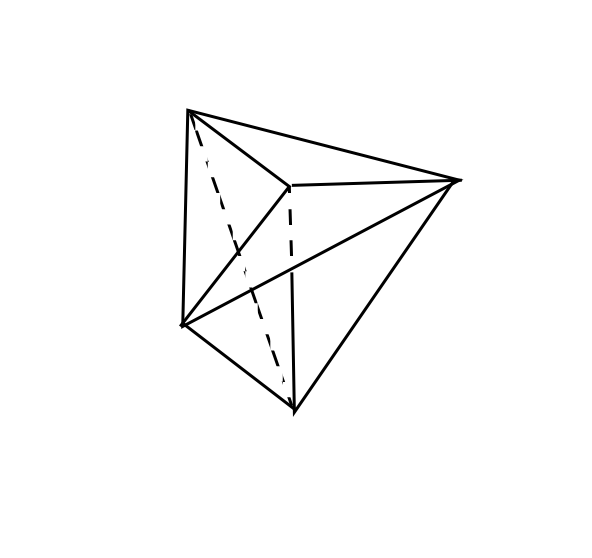}
		(c)
	\caption{For each $p \in \{v \in V(S(X_k)) \mid deg_{S(X_k)}(v) > 2\} \setminus X_{k+1}$ ($k \in \{2,\ldots,n\}$),
	(a), (b), (c) describes $X_k,X_{k-1},X_{k-2}$ near $p$.}
\end{figure}

\begin{defn}[Inscribed map]\label{inscribed map}\rm
	For a closed oriented surface $\Sigma$,
	let $f: \Sigma \to \mathbb{R}^{3}$ be an immersion.
	Assume $n = \max_{x \in \mathbb{R}^{3} \setminus f(\Sigma)}\omega(f,x)$,
	and $\{(\tilde{X}_1,X_1),\ldots,(\tilde{X}_n,X_n)\} \in I(\zeta)$.
	For each $k \in \{1,2,\ldots,n\}$,
	let $g_k: D_k \to \mathbb{R}^{3}$ be an embedding such that $g_k(D_k) = D_k(f)$,
	and let $A_k = g_{k}^{-1}(X_k)$, $B_k = g_{k}^{-1}(X_{k+1})$ ($X_{n+1} = \emptyset$).
	Let $g: \coprod_{k=1}^{n} D_k \to \mathbb{R}^{3}$ be a map such that $g \mid_{D_k} = g_k, \forall k \in \{1,2,\ldots,n\}$.
	We obtain a map $g_1: M \to \mathbb{R}^{3}$ by following procedure:
	
	$\bullet$
	We cut off $A_k \cup B_k$ from $D_k$ to obtain a space $D^{'}_{k}$
	($k \in \{1,\ldots,n\}$).
	Assume $g_0: \coprod_{k=1}^{n} D^{'}_{k} \to \mathbb{R}^{3}$ is the map induced by $g$.
	For all $k \in \{2,\ldots,n\}$ and $\alpha$ a $2$-cell of $X_k$,
	assume $\alpha^{+}_{1},\alpha^{-}_{1}$ (respectively, $\alpha^{+}_{2},\alpha^{-}_{2}$) are the $2$ components of $(g_{0} \mid_{D_{k}^{'}})^{-1}(\overline{\alpha})$ (respectively, $(g_{0} \mid_{D_{k-1}^{'}})^{-1}(\overline{\alpha})$) which lie in the left and right side respectively.
	Let $h$ be the equivalence relation such that $x \stackrel{h}{\sim} y$ if
	there exists 
	$k \in \{2,\ldots,n\}$ and $\alpha$ a $2$-cell of $X_k$,
	$x \in \alpha_{1}^{+}, y \in \alpha_{2}^{-}, g_0(x) = g_0(y)$ or $x \in \alpha_{2}^{+}, y \in \alpha_{1}^{-}, g_0(x) = g_0(y)$.
	Let $M = \coprod_{k=1}^{n} D^{'}_{k} /\sim_h$,
	and $g_1: M \to \mathbb{R}^{3}$ is induced by $g_0$.
	
	We say $g_1$ is an \emph{inscribed map} of $f$
	\emph{associated} to $\{(\tilde{X}_1,X_1),\ldots,(\tilde{X}_n,X_n)\}$.
\end{defn}

\begin{lm}
	$M$ is a (compact, connected) $3$-manifold, and $g_1$ is an immersion.
\end{lm}

\begin{proof}
	We can verify that for each $p \in \mathbb{R}^{3}$,
	every $t \in g_{1}^{-1}(p)$ has an open neighborhood  homeomorphic to $\mathbb{R}^{3}$ or $\mathbb{R}^{3}_{+}$,
	and $g_1$ is a locally homeomorphism at $t$.
	We only explain this for the point $p$ that is a vertex of $S(X_k)$ with degree greater than $2$:
	If $p \in \{v \in V(S(X_k)) \mid deg_{S(X_k)}(v) > 2\}$ and
	$p \notin X_{k+1}$ ($k \in \{2,\ldots,n\}$),
	then there exists $N(p)$ an open neighborhood of $p$ in $\mathbb{R}^{3}$ such that
	$N(p) \cap X_k, N(p) \cap X_{k-1}, X(p) \cap X_{k-2}$ are homeomorphic to Figure \ref{N} (a), (b), (c),
	and $p \cap X_{k-3} = \emptyset$.
	So we can verify that every point in $g_{1}^{-1}(p)$ has an open neighborhood which is homeomorphic to $\mathbb{R}^{3}$
	and homeomorphically embedded into $\mathbb{R}^{3}$ by $g_1$.
	
	Moreover,
	if $p \in g_1(M) \subseteq \mathbb{R}^{3}$,
	assume $l \subseteq \mathbb{R}^{3}$ is a ray starting from $p$ and parallel to $x$-axis.
	For each $x \in l \setminus D_1(f)$, $g_{1}^{-1}(x) = \emptyset$.
	So every component of $g_{1}^{-1}(l)$ contains a point in $\partial M$.
	Then every point in $g_{1}^{-1}(p)$ is in the same connected component with $\partial M$.
	Hence $M$ is connected.
\end{proof}

In the following , we say an inscribed map of $f$ associated to $\{(\tilde{X}_1,X_1),\ldots,(\tilde{X}_n,X_n)\}$
is an \emph{extension} of $f$ \emph{related to} $\{(\tilde{X}_1,X_1),\ldots,(\tilde{X}_n,X_n)\}$.

\begin{ep}\rm
	Let $f: S^{2} \to \mathbb{R}^{3}$ be an immersion described by Figure \ref{example 1} (a).
	Figure \ref{example 1} (b) describes $D_3(f)$, $D_2(f)$, $D_1(f)$, $G_3(f)$, $G_2(f)$.
	Figure \ref{zeta of example 1} gives an inscribed set $\zeta = \{\{(\tilde{X}_1,X_1),(\tilde{X}_2,X_2),$ $(\tilde{X}_3,X_3)\}\}$ of $f$ ($\#(\zeta) = 1$).
	Then $\{(\tilde{X}_1,X_1),(\tilde{X}_2,X_2),(\tilde{X}_3,X_3)\} \in I(\zeta)$.
	Hence we can construct (exactly) one extension of $f$.
	And Figure \ref{inscribed map of example 1} shows the construction of this extension (the extension of $f$ related to $\{(\tilde{X}_1,X_1),(\tilde{X}_2,X_2),(\tilde{X}_3,X_3)\}$).
\end{ep}

\begin{figure}\label{example 1}
	\includegraphics[width=0.46\textwidth]{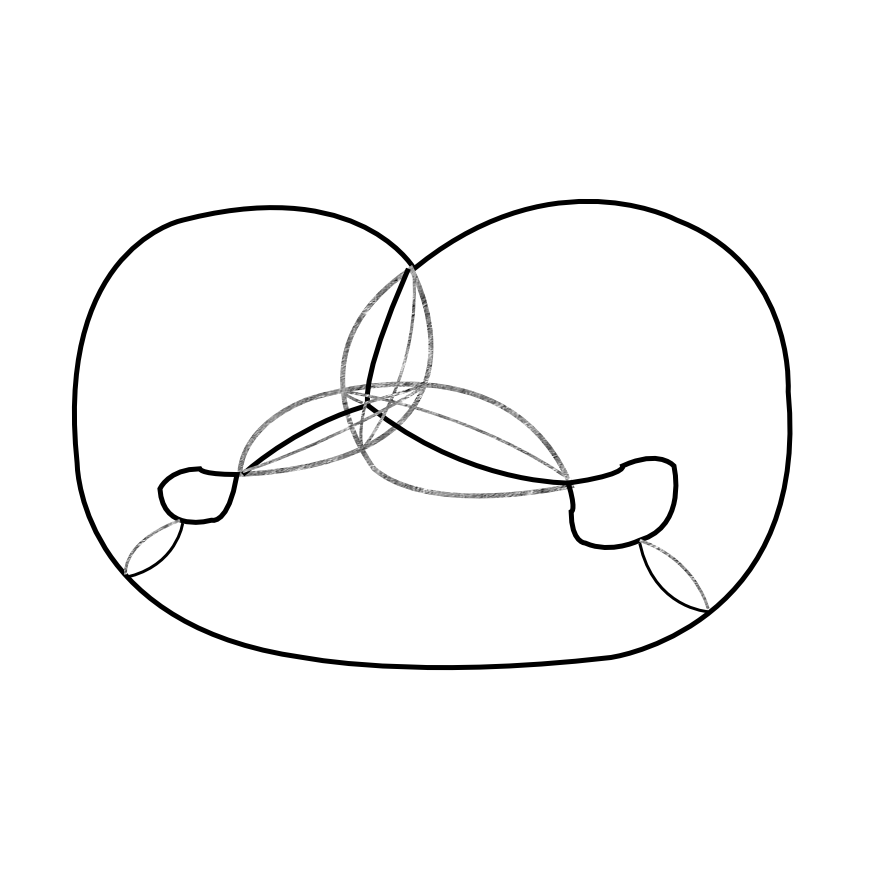}
	(a)
	\includegraphics[width=0.46\textwidth]{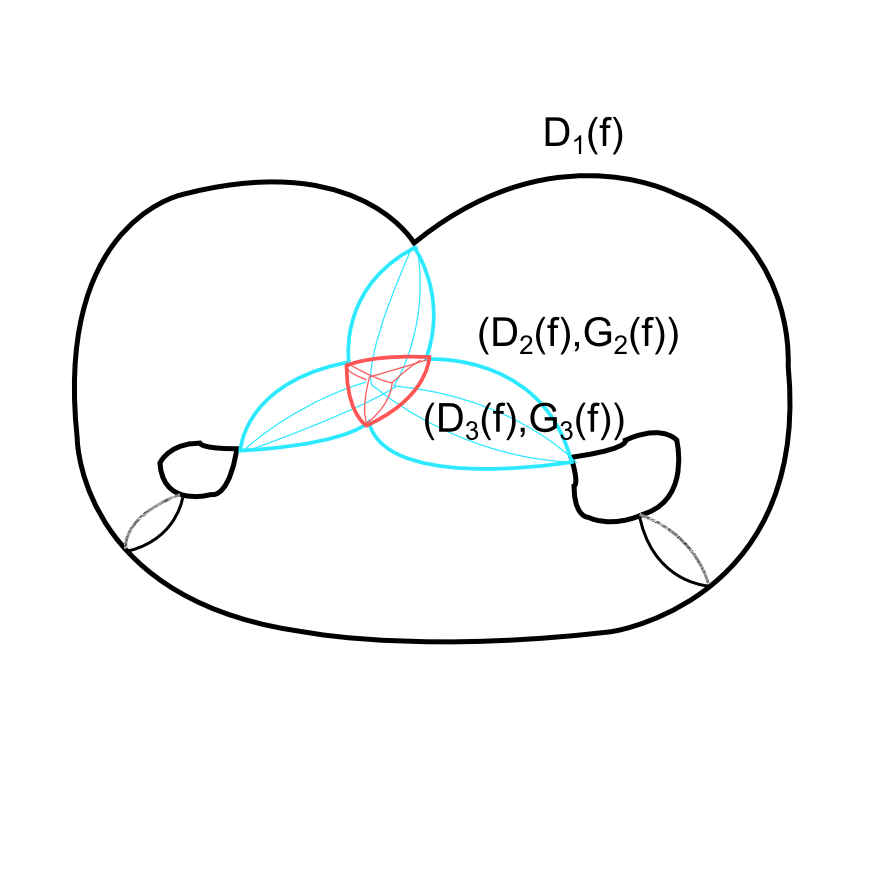}
	(b)
	\caption{$f: S^{2} \to \mathbb{R}^{3}$.}
\end{figure}

\begin{figure}\label{zeta of example 1}
	\includegraphics[width=0.33\textwidth]{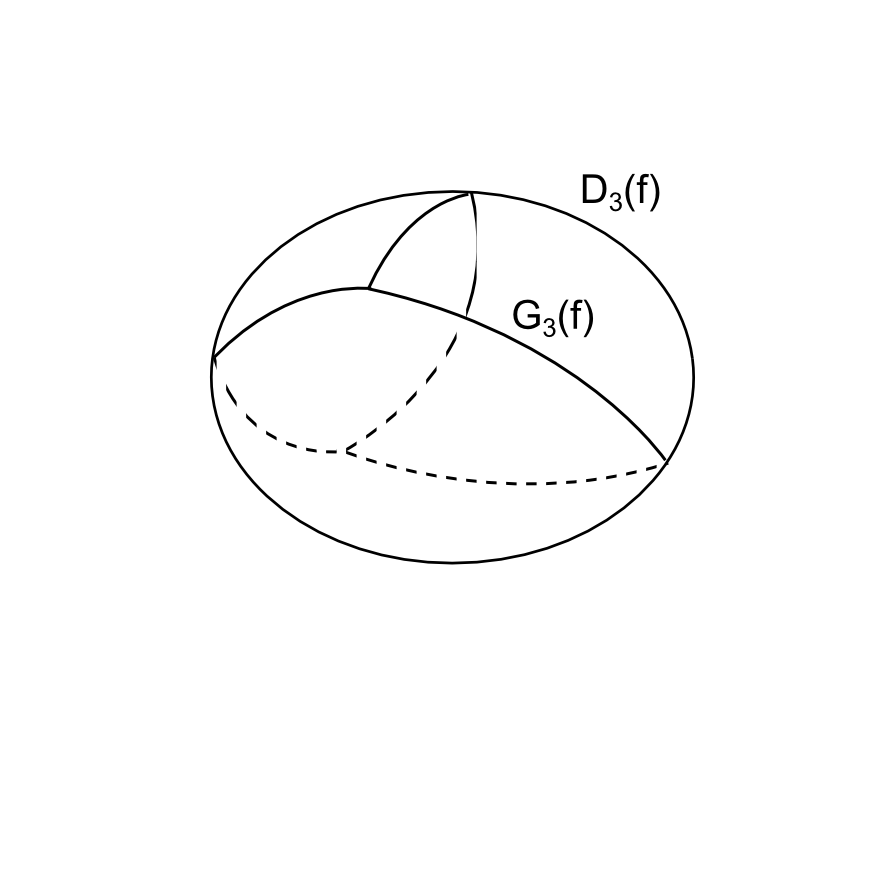}
	\includegraphics[width=0.33\textwidth]{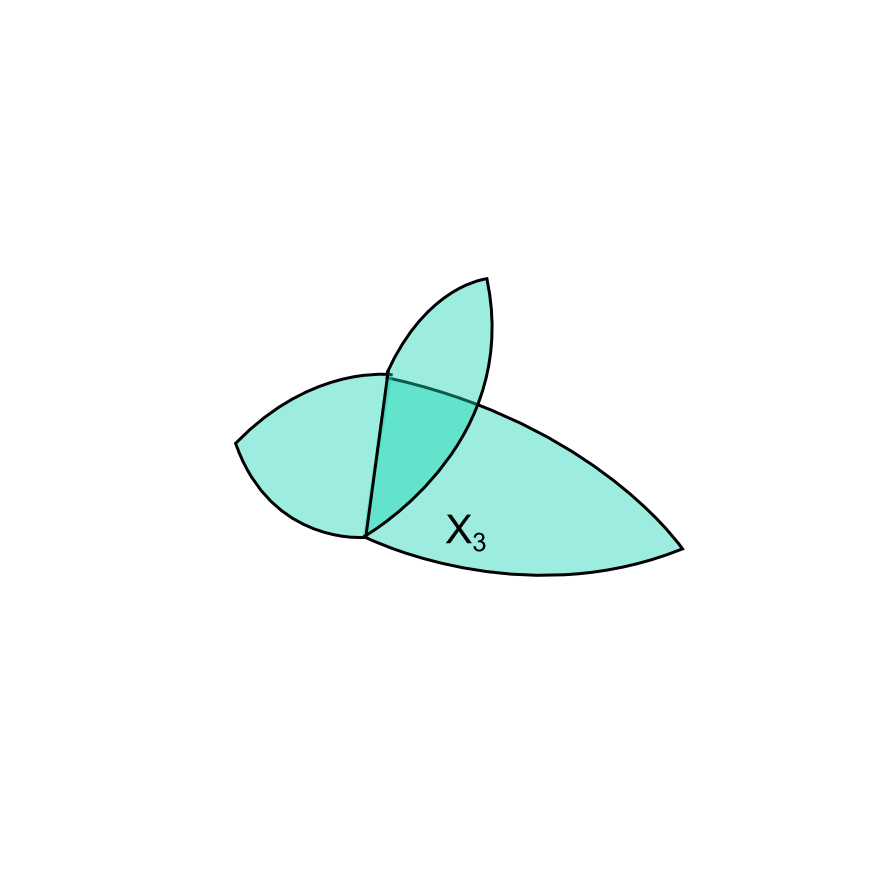}
	\includegraphics[width=0.33\textwidth]{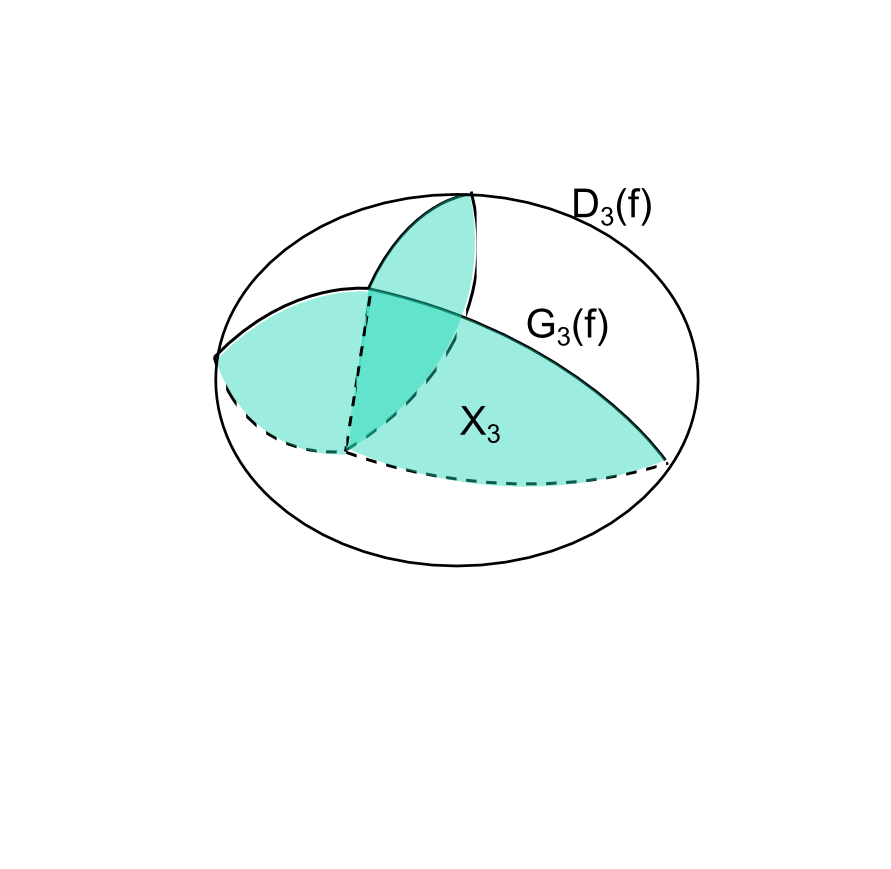}
	\includegraphics[width=0.33\textwidth]{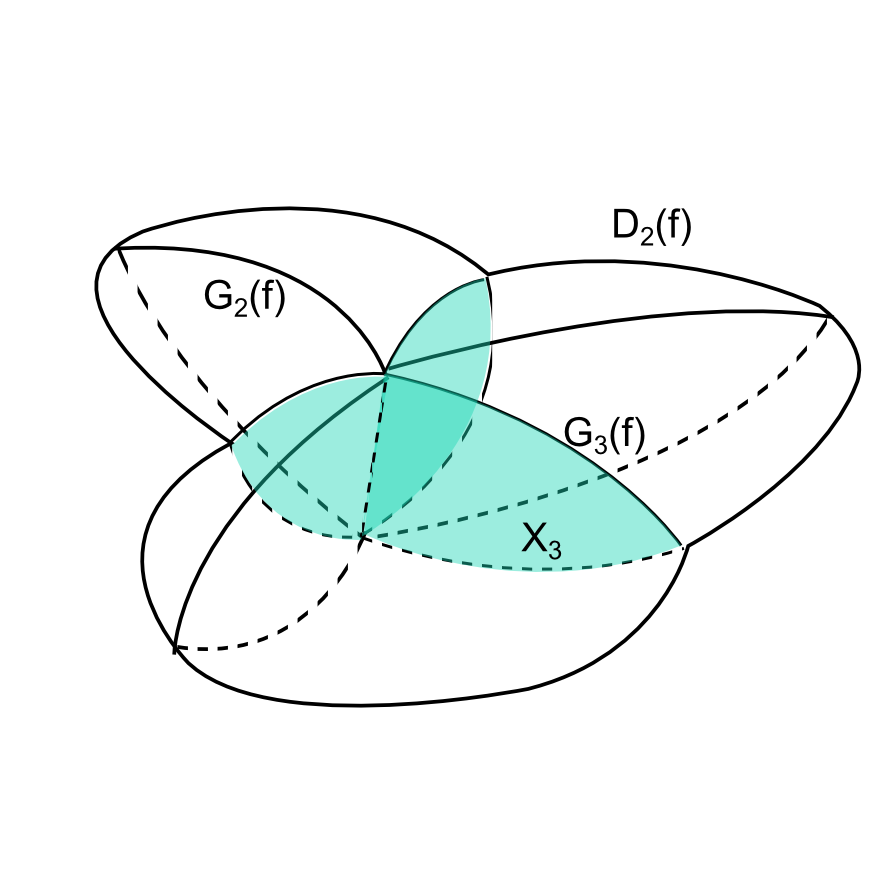}
	\includegraphics[width=0.33\textwidth]{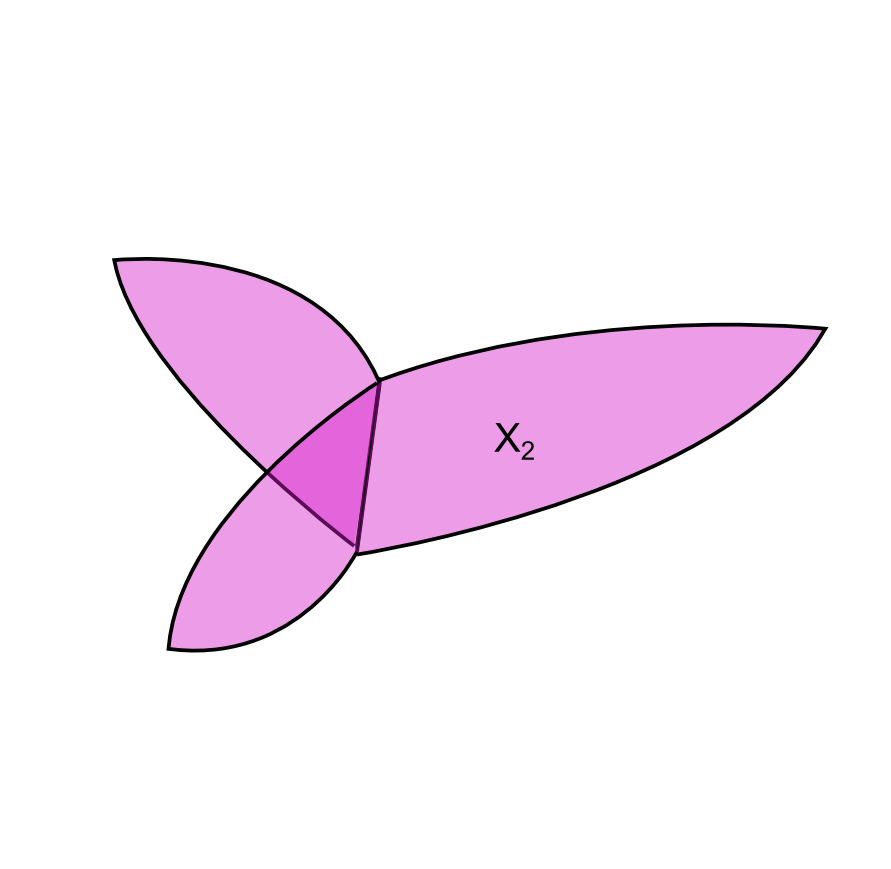}
	\includegraphics[width=0.33\textwidth]{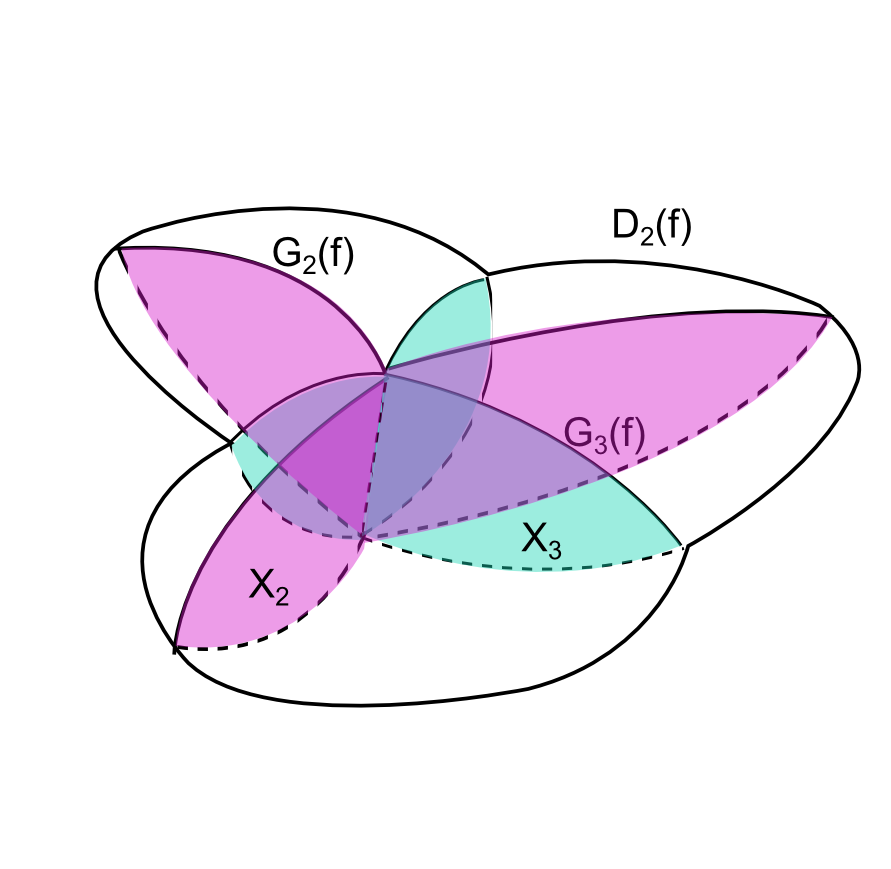}
	\includegraphics[width=0.5\textwidth]{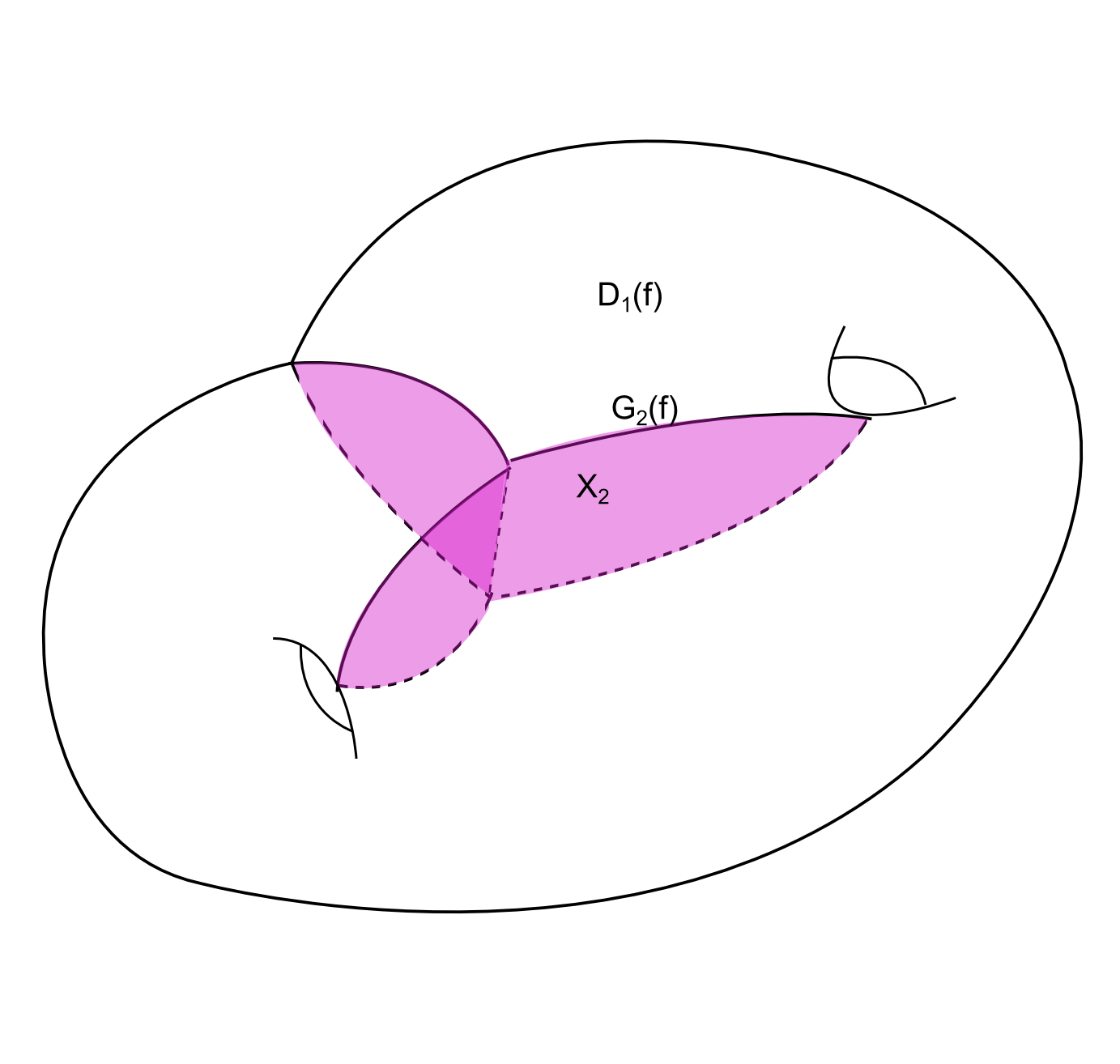}
	\includegraphics[width=0.5\textwidth]{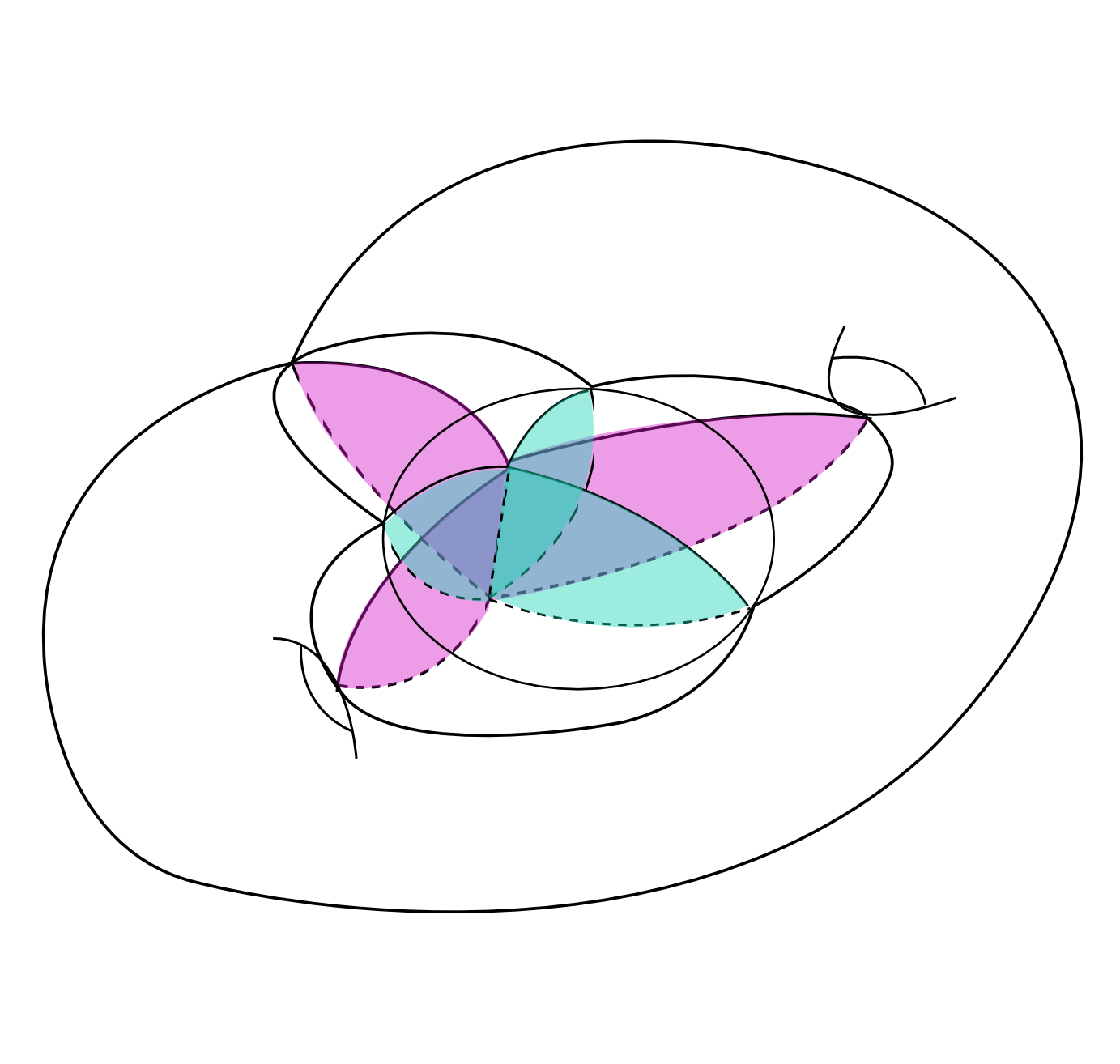}
	\caption{The inscribed set $\zeta = \{\{(\tilde{X}_1,X_1),(\tilde{X}_2,X_2),$ $(\tilde{X}_3,X_3)\}\}$ of $f$.}
\end{figure}

\begin{figure}\label{inscribed map of example 1}
	\begin{center}
	\includegraphics[width=0.49\textwidth]{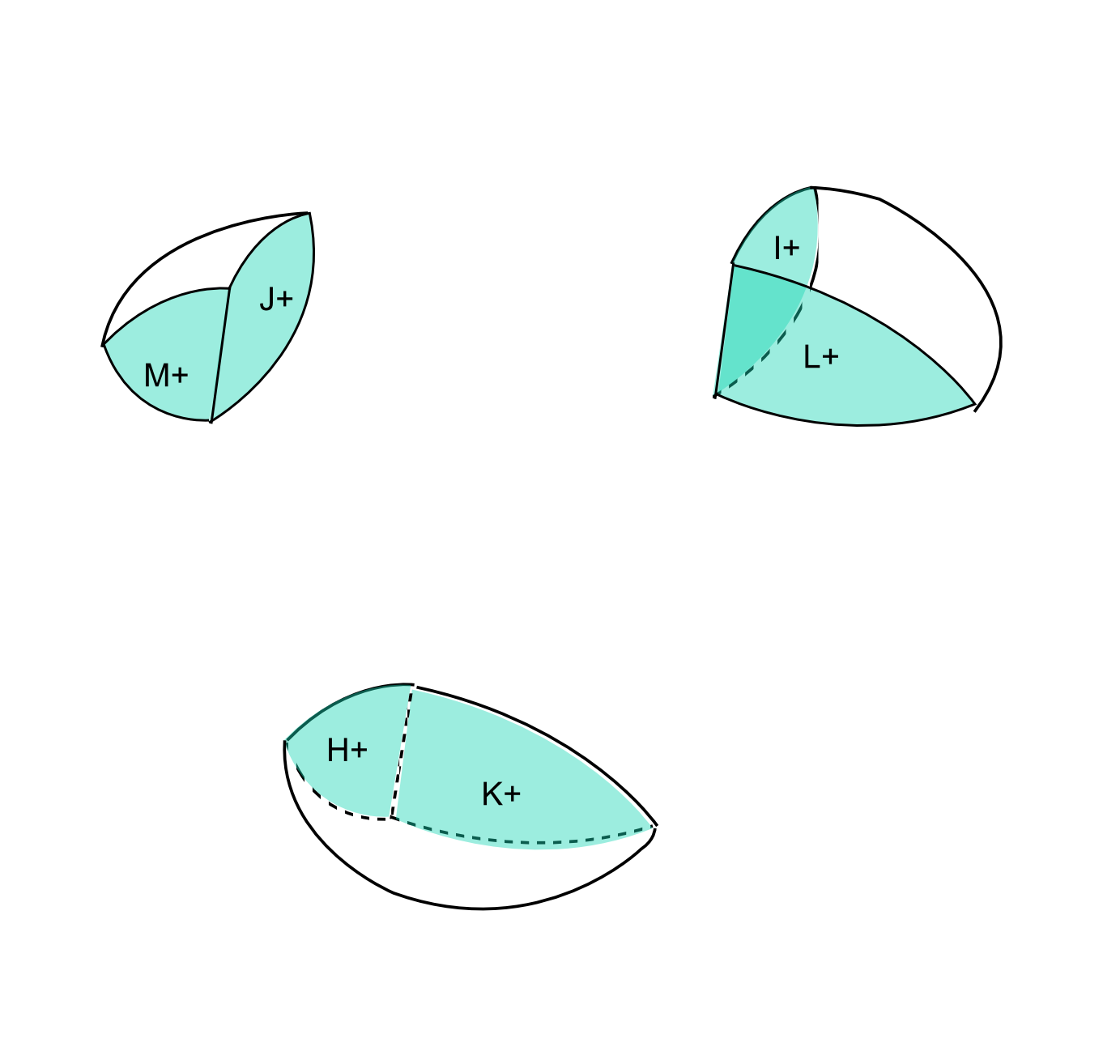}
	\includegraphics[width=0.49\textwidth]{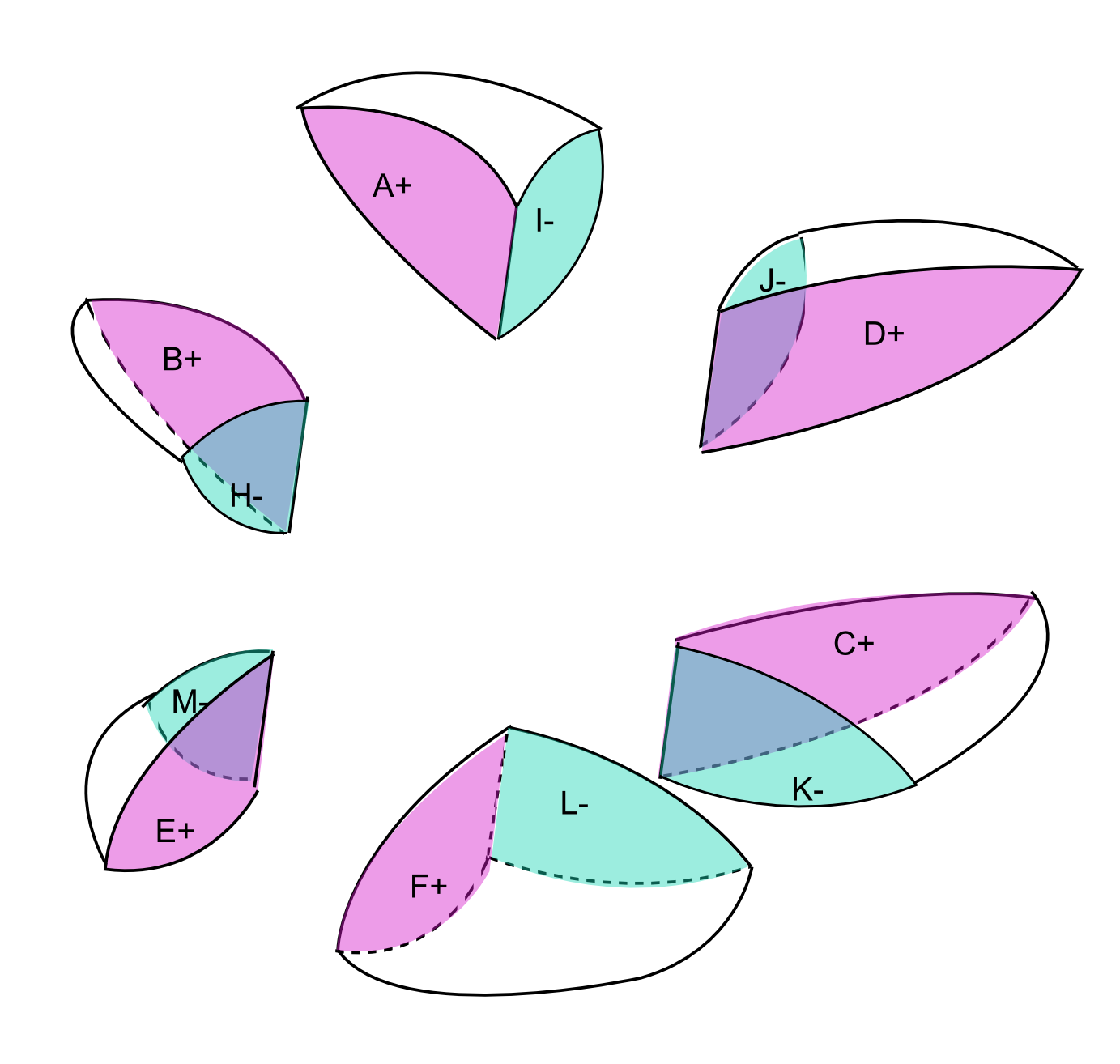}
	\includegraphics[width=0.49\textwidth]{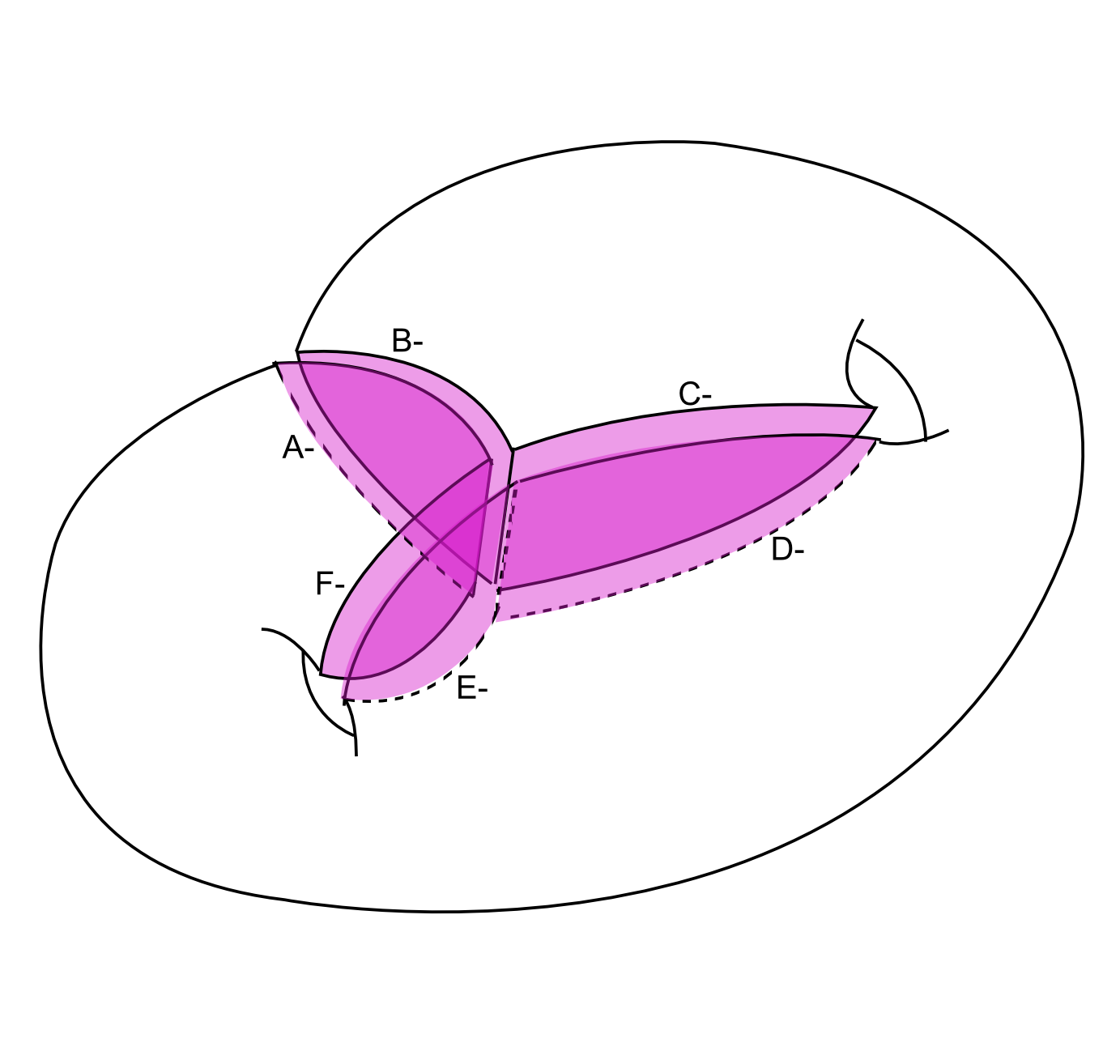}
	\end{center}
	\caption{The construction of the extension of $f$ related to $\{(\tilde{X}_1,X_1),(\tilde{X}_2,X_2),(\tilde{X}_3,X_3)\}$.}
\end{figure}

\section{The proof of Theorem 1}\label{section 5}\

In this section,
we will prove Theorem \ref{main theorem}.
If $f: \Sigma \to \mathbb{R}^{3}$ is an immersion of the closed oriented surface $\Sigma$
and $\zeta$ is an inscribed set of $f$,
then there exists a map $q: I(\zeta) \to E(f)$  (where $E(f)$ is the set of equivalence classes of immersions of $3$-manifolds to extend $f$)
sending each element of $I(\zeta)$ to the extension of $f$ related to it.
We prove that $q$ is injective in Lemma \ref{injective},
and we prove that $q$ is surjective in Proposition \ref{surjective}.

\begin{lm}\label{injective}
	For a closed oriented surface $\Sigma$,
	let $f: \Sigma \to \mathbb{R}^{3}$ be an immersion and $\zeta$ an inscribed set of $f$.
	Assume $n = \max_{x \in \mathbb{R}^{3} \setminus f(\Sigma)} \omega(f,x)$.
	If $\{(\tilde{X}_1,X_1),\ldots,(\tilde{X}_n,X_n)\},\{(\tilde{Y}_1,Y_1),\ldots,(\tilde{Y}_n,Y_n)\}$
	are $2$ different elements of $I(\zeta)$,
	then the $2$ extensions of $f$ related to $\{(\tilde{X}_1,X_1),\ldots,(\tilde{X}_n,X_n)\},\{(\tilde{Y}_1,Y_1),\ldots,$ $(\tilde{Y}_n,Y_n)\}$ are inequivalent.
\end{lm}

\begin{proof}
	The proof is similar to the proof of \hyperref[Zhao]{[9, Lemma 6.1]}.
	
	Assume $g_1: M_1 \to \mathbb{R}^{3},g_2: M_2 \to \mathbb{R}^{3}$ are the extensions related to 
	$\{(\tilde{X}_1,X_1),\ldots,(\tilde{X}_n,X_n)\}$,
	$\{(\tilde{Y}_1,Y_1),\ldots,(\tilde{Y}_n,Y_n)\}$.
	Then there exists $k \in \{2,3,\ldots,n\}$ such that
	$X_k \ne Y_k$ and $X_i = Y_i$ for each $k+1 \leqslant i \leqslant n$.
	Note that $\tilde{X}_k = \tilde{Y}_k$ (since $\tilde{X}_k$ is yielded by $\{(\tilde{X}_{k+1},X_{k+1}),\ldots,(\tilde{X}_n,X_n)\}$, and $\tilde{Y}_k$ is yielded by $\{(\tilde{Y}_{k+1},Y_{k+1}),\ldots,(\tilde{Y}_n,Y_n)\}$).
	So there exists $\alpha$ a $2$-cell of $\tilde{X}_k$ such that
	$\alpha$ is contained by exactly one of $X_k, Y_k$.
	Assume without loss of generality that $\alpha \subseteq X_k, \alpha \nsubseteq Y_k$.
	
	For each $\gamma$ a $2$-cell of $\tilde{X}_i$ ($i \in \{2,\ldots,n\}$),
	we denote by $D_+(\gamma)$ (respectively $D_-(\gamma)$) the closure of the component of  $D_i(f) \setminus (\tilde{X_i} \cup X_{i+1})$ which lie in the left side (respectively the right side) of $\gamma$.
	Then $\partial D_+(\gamma) \cap (X_{i+1} \cup \partial D_{i}(f)), \partial D_-(\gamma) \cap (X_{i+1} \cup \partial D_{i}(f)) \ne \emptyset$.
    
	There exist $m \in \{k,k+1,\ldots,n\}$ and $p_1 \in \partial D_m(f) \setminus G_m(f)$
    such that:
    for each $t \in \{1,\ldots,m-k-1\}$,
    $\exists$ $\alpha_t$ a $2$-cell in $X_{k+t}$ such that $\alpha_t \subseteq D_+(\alpha_{t-1})$,
    $\alpha_0 = \alpha$,
    and $p_1 \in D_+(\alpha_{m-k-1})$.
    Choose $s \in \alpha$.
    There exists an immersion $h_1: [0,1] \to \mathbb{R}^{3}$ such that
    $h_1(\frac{i}{m-k}) \in \alpha_i$,
    $[\frac{i}{m-k},\frac{i+1}{m-k}]$ is homeomorphically embedded into $D_+(\alpha_{i})$ by $h_1$ ($0 \leqslant i \leqslant m-k-1$),
    $h_1(0) = s$,
    $h_1(1) = p_1$.
    Similarly,
    there exists $q \in \{k,k+1,\ldots,n\}$ and $p_2 \in \partial D_q(f) \setminus G_q(f)$,
    such that: 
    for each $t \in \{1,\ldots,q-k-1\}$,
    $\exists$ $\beta_t$ a $2$-cell in $X_{k+t}$ such that $\beta_t \subseteq D_-(\beta_{t-1})$,
    $\beta_0 = \alpha$,
    and $p_2 \in D_-(\alpha_{q-k-1})$.
    Let $h_2: [0,1] \to \mathbb{R}^{3}$ be an immersion such that
    $h_2(\frac{i}{q-k}) \in \beta_i$,
    $[\frac{i}{q-k},\frac{i+1}{q-k}]$ is homeomorphically embedded into $D_-(\alpha_{i})$ by $h_2$ ($0 \leqslant i \leqslant q-k-1$),
    $h_2(0) = s$,
    $h_2(1) = p_2$.
    
    There exist embeddings $\tilde{h}_{1}: [0,1] \to M_1,\tilde{h}_{2}: [0,1] \to M_2$ such that $g_1 \circ \tilde{h}_1 = g_2 \circ \tilde{h}_2 = h_1$,
    $\tilde{h}_{1}(1) \in \partial M_1,\tilde{h}_{2}(1) \in \partial M_2$.
    And there exists embeddings $\tilde{h}_{3}: [0,1] \to M_1,\tilde{h}_{4}: [0,1] \to M_2$ such that $g_1 \circ \tilde{h}_3 = g_2 \circ \tilde{h}_4 = h_2$,
    $\tilde{h}_{3}(1) \in \partial M_1,\tilde{h}_{4}(1) \in \partial M_2$.
    Then $\tilde{h}_1(0) \ne \tilde{h}_3(0)$,
    $\tilde{h}_2(0) = \tilde{h}_4(0)$ (since $\alpha \subseteq X_k, \alpha \nsubseteq Y_k$).
    So there exists $\tilde{h}_2([0,1]) \cup \tilde{h}_4([0,1])$ a properly embedded arc of $M_2$ mapped to $h_1([0,1]) \cup h_2([0,1])$ by $g_2$,
    but there is no properly embedded arc of $M_1$ mapped to $h_1([0,1]) \cup h_2([0,1])$ by $g_1$. 
    Hence $g_1,g_2$ are inequivalent.
\end{proof}

\begin{prop}\label{surjective}
	For a closed oriented surface $\Sigma$,
	let $f: \Sigma \to \mathbb{R}^{3}$ be an immersion and $\zeta$ an inscribed set of $f$.
	Assume $n = \max_{x \in \mathbb{R}^{3} \setminus f(\Sigma)} \omega(f,x)$.
	If $g: M \to \mathbb{R}^{3}$ is an immersion of a compact $3$-manifold $M$ such that $g \mid_{\partial M} = f$,
	then there exists $\{(\tilde{X}_1,X_1),\ldots,(\tilde{X}_n,X_n)\} \in I(\zeta)$,
	such that $g$ is the extension of $f$ related to $\{(\tilde{X}_1,X_1),\ldots,(\tilde{X}_n,X_n)\}$.
\end{prop}

\begin{proof}
	First,
	we will construct a sequence of cancellation operations $(g,M) \leadsto (g_{n-1},K_{n-1}) 
	\leadsto
	(g_{n-2},K_{n-2}) \leadsto \ldots 
	\leadsto (g_1,K_1)$ in the following
	(where $K_j$ is a branched $3$-manifold,
	$g_j$ is a branched immersion, 
	$g_j(\partial K_j) = \bigcup_{i=1}^{j}\partial D_i(f)$,
	$K_1$ is a $3$-manifold and $g_1$ is an embedding):
	
	\emph{Step $1$.}
	By Lemma \ref{case 1 appropriate} (i), 
	$(R(g),G(g)) = (D_n(f),G_n(f))$ is appropriate.
	So $\zeta$ yields $\tilde{X}_n$ a $(D_n(f),G_n(f))$-simple $2$-complex.
	Assume $A_{(n,1)},\ldots,A_{(n,t_n)}$ are cancellable domains such that
	$\tilde{X}_n$ is the $2$-complex associated to them ($A_{(n,1)},\ldots,A_{(n,t_n)}$ exist and they are determined uniquely by $\tilde{X}_n$, 
	see Lemma \ref{case 1 appropriate} (ii) and Remark \ref{determine cancellable domains} (a)).
	Let $X_n = X(A_{(n,1)},\ldots,A_{(n,t_n)}) \in sub(X_n)$.
	Then $\{(X_n,\tilde{X}_n)\} \in \zeta_n$.
	
	We cancel $\{A_{(n,1)},\ldots,A_{(n,t_n)}\}$.
	The cancellation $(g,M) \stackrel{\{A_{(n,1)},\ldots,A_{(n,t_n)}\}}{\leadsto} (g_{n-1},K_{n-1})$ produces a branched immersion
	$g_{n-1}: K_{n-1} \to \mathbb{R}^{3}$ such that:
	
	\emph{Property} (a):
	$g_{n-1}(\partial K_{n-1}) = \bigcup_{i=1}^{n-1}\partial D_{i}(f)$.
	
	\emph{Property} (b):
	The embedded graph $S(X_n)$ is the branch set of $g_{n-1}$.
	For each $x \in S(X_n)$,
	$x$ has index $1$ if
	$x \in S(X_n) \setminus \{v \in V(S(X_n)) \mid deg_{S(X_n)}(v) > 2\}$,
	and $x$ has index $2$ if
	$x \in \{v \in V(S(X_n)) \mid deg_{S(X_n)}(v) > 2\}$.
	
	\emph{Property} (c):
	The cancellation is regular (Lemma \ref{case 1 appropriate} (ii)).
	Let $h_{n-1}: X_n \to K_{n-1}$ denote the associated map (Definition \ref{cancellation operation} (iii)) of the cancellation.
	Then $h_{n-1}(X_n \cap D_{n-1}(f)) = h_{n-1}(X_n) \cap \partial K_{n-1}$ (Remark \ref{regular means} (d)).
	And $h_{n-1}(S(X_n))$ is the singular set of $g_{n-1}$.
	
	\emph{Property} (d):
	For each $e \in E(S(X_n))$,
	there exists three $2$-cells $\alpha_1,\alpha_2,\alpha_3$ of $X_n$ 
	such that $e \subseteq \overline{\alpha_1}, \overline{\alpha_2}, \overline{\alpha_3}$,
	and assume that $\alpha_1,\alpha_2,\alpha_3$ are in clockwise.
	Assume $e_0 = h_{n-1}(e)$. 
	Assume $\beta_1,\ldots,\beta_6$ are the components of $g_{n-1}^{-1}(\alpha_1),g_{n-1}^{-1}(\alpha_2),g_{n-1}^{-1}(\alpha_3)$ 
	such that $e_0 \subseteq \overline{\beta_i}$ ($i = 1,\ldots,6$),
	and  $\beta_1,\ldots,\beta_6$ are in clockwise.
	Assume without loss of generality $h_{n-1}(\alpha_1) = \beta_1$ (then $\beta_1, \beta_4 \subseteq  g_{n-1}^{-1}(\alpha_1)$, $\beta_2, \beta_5 \subseteq  g_{n-1}^{-1}(\alpha_2)$, $\beta_3, \beta_6 \subseteq  g_{n-1}^{-1}(\alpha_3)$).
	Then $h_{n-1}(\alpha_2) = \beta_5$, $h_{n-1}(\alpha_3) = \beta_3$.
	
	\emph{Property} (e):
	$(D_{n-1}(f),X_n,G_{n-1}(f) \cup S(X_n))$ is appropriate. We explain this as follows:
	
	If $L$ is one of the components obtained by cutting off $X_n$ from $D_{n-1}(f)$,
	assume $i: L \to D_{n-1}(f)$ is the continuous map induced by the ``cutting off'' ($i \mid_{Int(L)}$ is an inclusion map, and $i(Int(L))$ is one of the components of $D_{n-1}(f) \setminus X_n$).
	Assume $S = \overline{\{x \in \partial L \mid i(x) \in Int(i(L))\}}$.
	Let $L_0$ be the space obtained by cutting off $g_{n-1}^{-1}(i(L)) \cap (h_{n-1}(S(X_n)) \cup g_{n-1}^{-1}(i(S)))$ from $g_{n-1}^{-1}(i(L))$,
	and assume $j: L_0 \to g_{n-1}^{-1}(i(L))$ is the continuous map induced by the ``cutting off'' ($j \mid_{Int(L_0)}$ is an inclusion map).
	Then $L_0$ is a $3$-manifold that may be disconnected,
	and there exists a covering map $g^{'}: L_0 \to L$ such that $g^{'} \mid_{g_{n-1}^{-1}(Int(i(L)))} = i^{-1} \circ g_{n-1} \circ j \mid_{g_{n-1}^{-1}(Int(i(L)))}$.
	For each component $A$ of $\partial L \setminus i^{-1}(G_{n-1}(f) \cup S(X_n))$,
	assume $A_0 = \overline{g^{'-1}(A) \cap j^{-1}(h_{n-1}(X_n) \cup \partial K_{n-1})}$,
	then $g^{'}$ maps $Int(A_0)$ homeomorphically to $A$.
	By Lemma \ref{appropriate = lifted},
	$(L, i^{-1}(G_{n-1}(f) \cup S(X_n)))$ is appropriate.
	So $(D_{n-1}(f),X_n,G_{n-1}(f) \cup S(X_n))$ is appropriate.
	
	\emph{Step $k+1$.}
	After $k$ steps,
	assume we obtain a branched immersion $g_{n-k}: K_{n-k} \to \mathbb{R}^{3}$ 
	from a sequence of cancellation operation
	\begin{center}
	$(g,M) \stackrel{\{A_{(n,1)},\ldots,A_{(n,t_n)}\}}{\leadsto} (g_{n-1},K_{n-1})
\stackrel{\{A_{(n-1,1)},\ldots,A_{(n-1,t_{n-1})}\}}{\leadsto} \ldots$
$\stackrel{\{A_{(n-k+1,1)},\ldots,A_{(n-k+1,t_{n-k+1})}\}}{\leadsto} (g_{n-k},K_{n-k})$
	\end{center}
	and the following hold:
	
	\emph{Induction hypothesis} (i):
	$g(\partial K_{n-k}) = \bigcup_{i=1}^{n-k} \partial D_{i}(f)$ (then $R(g_{n-k}) = D_{n-k}(f)$, $G(g_{n-k}) = G_{n-k}(f)$).
	
	\emph{Induction hypothesis} (ii):
	$B_{n-k} = \overline{S(X_{n-k+1}) \setminus S(X_{n-k+2})}$ is the branched set of $g_{n-k}$, and $B_{n-k}$ is an embedded graph.
	For each $x \in B_{n-k}$,
	$x$ has index $2$ if $x$ is a vertex of $S(X_{n-k+1})$ with degree greater than $2$ and $x \notin S(X_{n-k+2})$,
	otherwise $x$ has degree $1$.
	
	\emph{Induction hypothesis} (iii):
	There exists $h_{n-k}: X_{n-k+1} \to K_{n-k}$ an embedding such that
	$g_{n-k} \circ h_{n-k} = id$,
	$h_{n-k}(X_{n-k+1} \cap \partial D_{n-k}(f)) = h_{n-k}(X_{n-k+1}) \cap \partial K_{n-k}$,
	and $h_{n-k}(B_{n-k})$ is the singular set of $g_{n-k}$.
	
	\emph{Induction hypothesis} (iv):
	For each $e \in E(B_{n-k})$,
	assume $\alpha_1,\alpha_2,\alpha_3$ are the three $2$-cells of $X_{n-k+1}$ 
	such that $e \subseteq \overline{\alpha_1}, \overline{\alpha_2}, \overline{\alpha_3}$,
	and $\alpha_1,\alpha_2,\alpha_3$ are in clockwise.
	Assume $e_0 = h_{n-k}(e)$. 
	Assume $\beta_1,\ldots,\beta_6$ are the components of $g_{n-k}^{-1}(\alpha_1),g_{n-k}^{-1}(\alpha_2),g_{n-k}^{-1}(\alpha_3)$ 
	such that $e_0 \subseteq \overline{\beta_i}$ ($i = 1,\ldots,6$),
	and $\beta_1,\ldots,\beta_6$ are in clockwise.
	Assume without loss of generality $h_{n-k}(\alpha_1) = \beta_1$ (then $\beta_1, \beta_4 \subseteq  g_{n-k}^{-1}(\alpha_1)$, $\beta_2, \beta_5 \subseteq  g_{n-k}^{-1}(\alpha_2)$, $\beta_3, \beta_6 \subseteq  g_{n-k}^{-1}(\alpha_3)$).
	Then $h_{n-k}(\alpha_2) = \beta_5$, $h_{n-k}(\alpha_3) = \beta_3$.
	
	\emph{Induction hypothesis} (v):
	$(D_{n-k}(f),X_{n-k+1},G_{n-k}(f) \cup B_{n-k})$ is appropriate.
	
	By Corollary \ref{exists-thin regular},
	$\{(\tilde{X}_{n-k+1},X_{n-k+1}),\ldots(\tilde{X}_n,X_n)\}$ yields a $(D_{n-k}(f),X_{n-k+1},G_{n-k}(f) \cup B_{n-k})$-simple $2$-complex $\tilde{X}_{n-k}$.
	Let $A_{(n-k,1)},\ldots,A_{(n-k,t_{n-k})}$ be the cancellable domains such that:
	$\tilde{X}_{n-k}$ is the $2$-complex associated to them,
	and for each $j \in \{1,2,\ldots,t_{n-k}\}$,
	$A_{(n-k,j)} \cap (\partial M_{n-k} \cup h_{n-k}(X_{n-k+1})) \ne \emptyset$
	($A_{(n-k,1)},\ldots,A_{(n-k,t_{n-k})}$ exist, similar to Lemma \ref{case 1 appropriate} (ii);
	and $A_{(n-k,1)},\ldots,$ $A_{(n-k,t_{n-k})}$ are uniquely determined, see Remark \ref{determine cancellable domains} (b)).
	Let $X_{n-k} = X(A_{(n-k,1)},\ldots,A_{(n-k,t_{n-k})})$.
	Because of induction hypothesis (iv),
	$X_{n-k} \in sub_{X_{n-k+1}}(\tilde{X}_{n-k})$.
	So $\{(\tilde{X}_{n-k},X_{n-k}),(\tilde{X}_{n-k+1},X_{n-k+1}),$ $\ldots,(\tilde{X}_n,X_n)\} \in \zeta_{n-k}$.
	
	We cancel $\{A_{(n-k,1)},\ldots,A_{(n-k,t_{n-k})}\}$.
	The cancellation $(g_{n-k},K_{n-k})
	\stackrel{\{A_{(n-k,1)},\ldots,A_{(n-k,t_{n-k})}\}}{\leadsto} (g_{n-k-1},$ $K_{n-k-1})$
	gives a branched immersion $g_{n-k-1}: K_{n-k-1} \to \mathbb{R}^{3}$.
	First,
	we verify that the cancellation is regular:
	
	$\bullet$
	If $x \in S(X_{n-k}) \cap S(X_{n-k+1})$,
	then $x \in B_{n-k}$ and we denote by $y$ the singular point of $g_{n-k}$ mapped to $x$.
	Then $y \in A_{(n-k,j)}$ if $x \in g_{n-k}(A_{(n-k,j)})$ ($\forall j \in \{1,2,\ldots,t_{n-k}\}$).
	Hence the cancellation of $\{A_{(n-k,1)},\ldots,A_{(n-k,t_{n-k})}\}$ is regular at $x$.
	
	$\bullet$
	Similar to Lemma \ref{case 1 appropriate} (ii),
	if $x \in S(X_{n-k}) \setminus S(X_{n-k+1})$,
	then the cancellation of $\{A_{(n-k,1)},\ldots,$ $A_{(n-k,t_{n-k})}\}$ is regular at $x$.
	 
	So the cancellation $(g_{n-k},K_{n-k})
	\stackrel{\{A_{(n-k,1)},\ldots,A_{(n-k,t_{n-k})}\}}{\leadsto} (g_{n-k-1},$ $K_{n-k-1})$
	is regular.
	Moreover, the following hold:
	
	\emph{Property} (a):
	$g(\partial K_{n-k-1}) = \bigcup_{i=1}^{n-k-1} \partial D_{i}(f)$.
	
	\emph{Property} (b):
	The branch points of $g_{n-k}$ with index $1$ are not branch points of $g_{n-k-1}$,
	and the branch points of $g_{n-k}$ with index $2$ are branch points of $g_{n-k-1}$ with index $1$.
	So the branch set $B_{n-k-1}$ of $g_{n-k-1}$ is $\overline{S(X_{n-k}) \setminus S(X_{n-k+1})}$,
	and for each $x \in B_{n-k-1}$,
	$x$ has index $2$ if $x \in \{v \in V(S(X_{n-k})) \mid deg_{S(X_{n-k})}(v) > 2\} \setminus S(X_{n-k+1})$,
	and $x$ has index $1$ otherwise. 
	
	\emph{Property} (c):
	Since the cancellation $(g_{n-k},K_{n-k})$
	$\stackrel{\{A_{(n-k,1)},\ldots,A_{(n-k,t_{n-k})}\}}{\leadsto} (g_{n-k-1},$ $K_{n-k-1})$
	is regular,
	there exists $h_{n-k-1}: X_{n-k} \to K_{n-k-1}$ the associated map of the cancellation.
	$h_{n-k-1}(X_{n-k} \cap \partial D_{n-k-1}(f)) = h_{n-k-1}(X_{n-k}) \cap \partial K_{n-k-1}$ (Remark \ref{regular means} (d)).
	And $h_{n-k-1}$ maps $B_{n-k-1}$ homeomorphically to the singular set of $g_{n-k-1}$.
	
	\emph{Property} (d):
	\emph{Induction hypothesis} (iv) is developed for $X_{n-k}$ and $B_{n-k-1}$.
	We state this again as follows:
	
	For each $e \in E(B_{n-k-1})$,
	there exists $\alpha_1,\alpha_2,\alpha_3$ the three $2$-cells of $X_{n-k}$ 
	such that $e \subseteq \overline{\alpha_1}, \overline{\alpha_2}, \overline{\alpha_3}$,
	and $\alpha_1,\alpha_2,\alpha_3$ are in clockwise.
	Assume $e_0 = h_{n-k-1}(e)$. 
	Assume $\beta_1,\ldots,\beta_6$ are the components of $g_{n-k-1}^{-1}(\alpha_1),g_{n-k-1}^{-1}(\alpha_2),g_{n-k-1}^{-1}(\alpha_3)$ 
	such that $e_0 \subseteq \overline{\beta_i}$ ($i = 1,\ldots,6$),
	and $\beta_1,\ldots,\beta_6$ are in clockwise.
	Assume without loss of generality $h_{n-k-1}(\alpha_1) = \beta_1$ (then $\beta_1, \beta_4 \subseteq  g_{n-k-1}^{-1}(\alpha_1)$, $\beta_2, \beta_5 \subseteq  g_{n-k-1}^{-1}(\alpha_2)$, $\beta_3, \beta_6 \subseteq  g_{n-k-1}^{-1}(\alpha_3)$).
	Then $h_{n-k-1}(\alpha_2) = \beta_5$, $h_{n-k-1}(\alpha_3) = \beta_3$.
	
	\emph{Property} (e):
	Similar to the \emph{Property} (e) of \emph{Step $1$},
	$(D_{n-k-1}(f),X_{n-k},G_{n-k-1}(f) \cup B_{n-k-1})$ is appropriate.
	
	Hence we can verify that all induction hypothesises will be developed in the next step (\emph{Step $k+2$}).
	
	In the end,
	we have constructed $\{(\tilde{X}_1,X_1),\ldots,(\tilde{X}_n,X_n)\} \in \zeta$
	with a sequence of cancellation operations
	\begin{center}
		$(g,M) \stackrel{\{A_{(n,1)},\ldots,A_{(n,t_n)}\}}{\leadsto} (g_{n-1},K_{n-1})
		\stackrel{\{A_{(n-1,1)},\ldots,A_{(n-1,t_{n-1})}\}}{\leadsto} \ldots
		\stackrel{\{A_{(2,1)},\ldots,A_{(2,t_2)}\}}{\leadsto} (g_1,K_1)$.
	\end{center}
    Note that
    $\#(g_{k-1}^{-1}(x)) + index_{g_{k-1}}(x) = \#(g_{k}^{-1}(x)) + index_{g_k}(x) -1$ for each
    $x \in D_k(x)$
    (where $index_{g_j}(x)$ is the index of $x$ if $x$ is a branch point of the map $g_j$, and $index_{g_j}(x) = 0$ if $x$ is not a branch point of the map $g_j$; and $g_n = g$).
    So $g_1: K_1 \to \mathbb{R}^{3}$ is an embedding.
    Then $X_1 = \emptyset$.
	Hence $\{(\tilde{X}_1,X_1),\ldots,(\tilde{X}_n,X_n)\} \in I(\zeta)$.
	
	From this sequence of cancellation operations,
	we can verify that $g: M \to \mathbb{R}^{3}$ is the inscribed map of $f$ associated to $\{(\tilde{X}_1,X_1),\ldots,(\tilde{X}_n,X_n)\}$,
	i.e. $g$ is the extension of $f$ related to $\{(\tilde{X}_1,X_1),\ldots,(\tilde{X}_n,X_n)\}$.
\end{proof}

\section{Acknowledgments}\label{section 6}\
This paper is the motivation for the author to write \hyperref[Zhao]{[9]}.
The author is grateful for Professor Shicheng Wang, Professor Yi Liu, Professor Jiajun Wang in these works,
and in my study.
Especially, the author thanks Professor Jiajun Wang for his comment that help me a lot.

\end{document}